\documentclass[10pt]{article}
\usepackage{geometry}                
\usepackage{graphicx}
\usepackage{fullpage}
\usepackage{amsmath}
\usepackage{amssymb}
\usepackage{amsthm}
\usepackage{url}
\usepackage{epsfig}
\usepackage{color}
\usepackage{refcount}
\usepackage{verbatim}
\usepackage{enumerate}
\usepackage{multirow}
\usepackage{framed}
\usepackage{slashbox}
\usepackage{caption}
\usepackage{subcaption}
\usepackage[normalem]{ulem}

\usepackage{fancyhdr}
\usepackage{datetime}

\usepackage{array}
\newcolumntype{L}[1]{>{\raggedright\let\newline\\\arraybackslash\hspace{0pt}}m{#1}}
\newcolumntype{C}[1]{>{\centering\let\newline\\\arraybackslash\hspace{0pt}}m{#1}}
\newcolumntype{R}[1]{>{\raggedleft\let\newline\\\arraybackslash\hspace{0pt}}m{#1}}

\usepackage{xr}

\usepackage[breaklinks=true,colorlinks,citecolor=blue,linkcolor=blue]{hyperref}
\usepackage[square,numbers,sort&compress]{natbib}

\usepackage{calc}

\usepackage{tikz}

\usepackage{color}
\definecolor{clemson-orange}{RGB}{234,106,32}
\definecolor{chicago-maroon}{RGB}{128,0,0}
\definecolor{cincinnati-red}{RGB}{190,0,0}
\definecolor{soft-cyan}{RGB}{68,85,90}
\usepackage{fullpage}
\usepackage{multicol}

\usepackage[utf8]{inputenc}

\newcommand{\bb}{\mathbb}
\newcommand{\R}{\bb R}
\newcommand{\Z}{{\bb Z}}
\newcommand{\N}{{\bb N}}

\theoremstyle{definition}
\newtheorem{theorem}{Theorem}[section]
\newtheorem{lemma}[theorem]{Lemma}
\newtheorem{corollary}[theorem]{Corollary}

\newtheorem{definition}[theorem]{Definition}
\newtheorem{remark}[theorem]{Remark}

\newtheorem{example}[theorem]{Example}
\newtheorem{conj}[theorem]{Conjecture}

\DeclareMathOperator*{\conv}{conv}

\DeclareMathOperator*{\argmax}{argmax}
\DeclareMathOperator*{\argmin}{argmin}

\DeclareMathOperator*{\vol}{vol}
\DeclareMathOperator*{\comp}{comp}
\DeclareMathOperator*{\icomp}{icomp}
\DeclareMathOperator*{\poly}{poly}
\DeclareMathOperator*{\size}{size}

\newcommand{\I}{\mathcal{I}}

\newcommand{\C}{\mathcal{C}}
\newcommand{\D}{\mathcal{D}}

\newcommand{\cQ}{\mathcal{Q}}
\newcommand{\cA}{\mathcal{A}}

\newcommand{\CP}{\mathcal{CP}}

\numberwithin{equation}{section}

\title{Complexity of optimizing over the integers}
\author{Amitabh Basu\thanks{Department of Applied Mathematics and Statistics, Johns Hopkins University, {\tt basu.amitabh@jhu.edu.} The author gratefully acknowledges support from Air Force Office of Scientific Research (AFOSR) grant FA95502010341 and National Science Foundation (NSF) grants CCF2006587.}\bigskip}
\date{\today}
\begin{document}

\maketitle

\begin{abstract} 
In the first part of this paper, we present a unified framework for analyzing the algorithmic complexity of any optimization problem, whether it be continuous or discrete in nature. This helps to formalize notions like ``input", ``size" and ``complexity" in the context of general mathematical optimization, avoiding context dependent definitions which is one of the sources of difference in the treatment of complexity within continuous and discrete optimization. In the second part of the paper, we employ the language developed in the first part to study information theoretic and algorithmic complexity of {\em mixed-integer convex optimization}, which contains as a special case continuous convex optimization on the one hand and pure integer optimization on the other. We strive for the maximum possible generality in our exposition. 

We hope that this paper contains material that both continuous optimizers and discrete optimizers find new and interesting, even though almost all of the material presented is common knowledge in one or the other community. We see the main merit of this paper as bringing together all of this information under one unifying umbrella with the hope that this will act as yet another catalyst for more interaction across the continuous-discrete divide. In fact, our motivation behind Part I of the paper is to provide a common language for both communities.
\end{abstract}

\part{\Large A general framework for complexity in optimization}

\section{The setup}
This paper deals with theoretical complexity analyses for optimization problems where some or all decision variables are constrained to take integer values. This means that we will look at provable upper and lower bounds on the efficiency of algorithms that solve these problems. We will consider the standard Turing machine model of computation, and we will study algorithms that receive an optimization problem as ``input" and we wish to study the efficiency of the algorithm as a function of the ``size" of the problem. Moreover, we wish to develop a framework that can seamlessly handle discrete optimization and continuous optimization. In particular, the focus of this paper will be on so-called {\em mixed-integer convex optimization}, which includes as special cases continuous, nonlinear convex optimization on the one hand and integer linear optimization on the other which can model problems with a combinatorial and discrete nature. Therefore, it is imperative to have a general setup that can formally make sense of ``writing down" an optimization problem as ``input" to an algorithm and the related ``size" of the optimization problem, no matter whether it is combinatorial or numerical in nature. 

For instance, most optimization problems with a combinatorial or discrete nature have a well defined, universally accepted notion of ``encoding" and therefore of ``size" (such as the Matching Problem or Traveling Salesperson Problem (TSP) on graphs, or linear programming (LP) over the rationals). Things are more tricky with continuous optimization problems. What is the ``size" of a general optimization problem of the form $\min\{f(x): g_i(x) \leq 0,\;\; i=1, \ldots, m\}$, where $I$ and $g_i$ are smooth, nonlinear functions defined on $\R^n$? If these functions have a particular algebraic form (e.g., polynomial optimization), then there is usually no controversy because one considers the encoding of the coefficients appearing in the polynomials. What if the functions are given via evaluation and gradient oracles?

In fact, the question can be raised whether the Turing machine model is even appropriate in this setting, because it only allows for computation over a finite set of symbols and numerical/scientific computation problems may require a more general model of computation. Several proposals have been put forward to address this issue; see~\cite{turing1937computable,blum1989theory,borodin-munro,friedman1971algorithmic,ko1991complexity,pour1983computability,abramson1971effective,lovasz1986algorithmic} as a representative list. It turns out that all of the discussion in this Part I of this paper will hold true in any of these different models of computation. For concreteness, we will assume a computational framework  that works within the standard framework of the Turing machine model of computation augmented with appropriate oracles~\cite{lovasz1986algorithmic,ko1991complexity} and has its roots in the constructive philosophy of mathematics~\cite{beeson2012foundations,bishop1967foundations}. We include the formal definition here for completeness of this article and the reader is referred to~\cite{lovasz1986algorithmic,ko1991complexity} for a fuller discussion.

\begin{definition}\label{def:real-oracles} An oracle implementing a real number $\alpha \in \R$ takes as input a rational number $\epsilon > 0$ and outputs a rational number $r$ such that $|r - \alpha| \leq \epsilon$. Moreover, there exists a constant $k$ (independent of $\epsilon$) such that the bit encoding size of $r$ is at most $k$ times the encoding size of $\epsilon$. The {\em size of the oracle} is $k$.

If the real number $\alpha$ implemented by an oracle is rational, then a second constant $k'$ is added to the size of the oracle with the guarantee that the bit encoding size of $\alpha$ is at most $k'$. 
\end{definition}
\bigskip

A study of general optimization methods in the 70s and 80s led to the insight that all such procedures (whether combinatorial or analytical in nature) can be understood in a unified framework which we present now. Our exposition is a summary of ideas put forth in several works and giving a complete survey of this literature is difficult. With apologies for our omissions, we list two references that we personally found to be most illuminating~\cite{Nemirovski_Yudin_book,traub1983information}.

The overall idea, roughly speaking, is that one gathers {\em information} about an optimization problem which tells the optimizer which instance needs to be solved within a problem class, and then computations are performed on the information gathered to arrive at a solution (possibly approximate, with guaranteed bounds on error). Let us formalize this idea in a way that encompasses both discrete and continuous optimization~\cite{traub1983information}.

\begin{definition}\label{def:gen-opt}[General optimization problem] An {\em optimization problem class} is given by a set $\I$ of {\em instances}, a set $G$ of {\em possible solutions}, and a {\em solution operator} $$S: \I \times \R_+ \to 2^G \cup\{INFEAS\} \cup \{UNBND\},$$ where $2^G$ denotes the power set of $G$ and the operator $S$ satisfies three properties:
\begin{enumerate}
\item $S(I,0) \neq \emptyset$ for all $I \in \I$, and
\item For any two nonnegative real numbers $\epsilon_1 <  \epsilon_2$ and any $I\in \I$, we have $S(I, \epsilon_1) \subseteq S(I, \epsilon_2)$.
\item If $INFEAS\in S(I,\epsilon)$ (or $UNBND\in S(I,\epsilon)$) for some $I\in \I$ and $\epsilon \geq 0$, then $S(I,\epsilon) = \{INFEAS\}$ (or $S(I,\epsilon) = \{UNBND\}$ respectively).
\item If $UNBND\in S(I,0)$ for some $I\in \I$, then $S(I,\epsilon) = \{UNBND\}$ for all $\epsilon\geq 0$.
\end{enumerate}

\end{definition}

The interpretation of the above definition is simple: $\I$ is the set of optimization problem instances we wish to study, $G$ is the space of solutions to the problems, and for any instance $I \in \I$ and $\epsilon \geq 0$, $S(I,\epsilon)$ is the set of $\epsilon$-approximate solutions with the understanding that $S(I,\epsilon) = \{INFEAS\}$ encodes the fact that $I$ is an infeasible instance (with respect to $\epsilon$ error; see examples below) and $S(I,\epsilon) = UNBND$ encodes the fact that the objective value for $I$ is unbounded. The advantage of this definition is that there is no need to assume any structure in the set $G$; for example, it could be some Euclidean space, or it could just as well be some combinatorial set like the set of edges in a graph. Linear Programming in $\R^n$ would set $G=\R^n$ while the Traveling Salesperson Problem on $n$ cities corresponds to setting $G$ to be all tours in the complete graph $K_n$. It is also not hard to encode optimization problems in varying ``dimensions" in this framework, e.g., $G$ is allowed to be $\bigcup_{n \in \N} \R^n$. Also, the notion of ``$\epsilon$-approximate" does not require any kind of norm or distance structure on $G$. Property 2 simply requires that as we allow more error, we obtain more solutions. Thus, Definition~\ref{def:gen-opt} captures discrete and continuous optimization in a clean, unified framework while allowing for a very flexible notion of an ``$\epsilon$-approximate" solution for $\epsilon > 0$.

\begin{example}\label{ex:concrete-ex} We now present some concrete examples.
\begin{enumerate}
\item Traveling Salesperson Problem (TSP). For any natural number $n\in \N$, the (symmetric) traveling salesperson problem for $n$ cities seeks to find a tour of minimum length that visits all cities, given pairwise intercity distances. To model this in the above framework, one defines $E_n$, for any natural number $n\in \N$, to be the set of all unordered pairs in $\{1, \ldots, n\}$. Let $G = \cup_{n\in \N} 2^{E_n}$, i.e., each element of $G$ is a subset of unordered pairs (these are edges in the tour). $\I$ is the family of all TSPs with given intercity distances. We allow any number $n$ of cities; of course, for a particular instance $I$ there is a fixed number of cities. $S(I,\epsilon)$ can be taken to be the set of all tours in the complete graph $K_n$ that are within an additive $\epsilon$ error or within a multiplicative $(1+\epsilon)$ factor of the optimal tour length in $I$. $\epsilon = 0$ corresponds to the set of all optimal tours.

One could also fix a natural number $n\in \N$ and simply consider only the problems on $n$ cities. In this case, $G = 2^{E_n}$ and $\I$ would consist of all possible tours on $n$ cities, i.e., where only the intercity distances are changed, but the number of cities is fixed.
\item Mixed-integer linear programming (MILP). 
\begin{itemize} \item (Fixed dimension) Let $n,d \in \N$ be fixed. $G = \R^n \times \R^d$, $\I$ is the set of all mixed-integer linear programs defined by matrices $A \in \R^{m\times n}, B \in \R^{m \times d}$, and vectors $b \in \R^m, c_1 \in \R^n, c_2 \in \R^d$, where $m\in \N$ can be chosen as any natural number (thus, $\I$ contains all MILPs with any number of constraints for $m=1,2, \ldots$, but the total number of variables is fixed): $$\max\{c_1^Tx + c_2^Ty \;:\; Ax + By \leq b,\;\; x \in \Z^n, y\in \R^d\}.$$ $S(I,\epsilon)$ may be defined to be all solutions $(x,y)\in G$ to an MILP instance $I$ such that $c_1^Tx + c_2^Ty$ is within an additive $\epsilon$ error of the optimal value for $I$. Taking $\epsilon =0$ would mean we are considering the exact optimal solution(s). Alternatively, one may define $S(I,\epsilon)$ to be the set of $(x,y)\in G$ such that there is an optimal solution to $I$ within $\epsilon$ distance to $(x,y)$.
\item (Variable dimension) We can consider the family of all MILPs with a fixed number of integer variables, but allowing for any number of continuous variables. Here $n \in \N$ is fixed and $G = \bigcup_{d \in \N}(\R^n \times \R^d)$. Everything else is defined as above. Similarly, we may also allow the number of integer variables to vary by letting $G = \bigcup_{n\in \N, d \in \N}(\R^n \times \R^d)$.
\end{itemize}
Fixing $n=0$ in the above settings would give us {\em (pure) linear programming}.
\item Nonlinear Optimization. In a similar fashion as above, we may model nonlinear optimization problems of the form \begin{equation}\label{eq:nonlinear-opt}\min\{f(x): g_i(x) \leq 0\; i=1, \ldots, m\}.\end{equation} The class $\I$ may restrict the structure of the objective and constraint functions (e.g., convex, twice continuously differentiable, nonsmooth etc.). As before, $S(I,\epsilon)$ may correspond to all solutions that are within an $\epsilon$ error of the true optimal value, or solutions within $\epsilon$ distance of the set of optimal solutions, or the set of points where the norm of the gradient is at most $\epsilon$ in the unconstrained case, or any other notion of $\epsilon$-approximate solutions commonly studied in the nonlinear optimization literature. One may also allow $\epsilon$ slack in the constraints, i.e., $g_i(x) \leq \epsilon$ for any $x\in S(I,\epsilon)$.
\end{enumerate}
\end{example}

Next, we discuss the notion of an oracle that permits us to figure out which problem instance we need to solve.

\begin{definition} An {\em oracle} for an optimization problem class $\I$ is given by a family $\cQ$ of possible {\em queries} and a set $H$ of possible {\em answers}. Each query $q\in \cQ$ is an operator $q: \I \to H$. We say that $q(I) \in H$ is the answer to the query $q$ on the instance $I\in \I$.
\end{definition}

\begin{example}\label{ex:concrete-oracles} We now consider some standard oracles for the settings considered in Example~\ref{ex:concrete-ex}.
\begin{enumerate}
\item For the TSP, the typical oracle uses two types of queries. One is the dimension query $q_{\dim}$, which returns the number $q_{\dim}(I)$ of cities in the instance $I$, and the queries $q_{ij}(I)$ which returns the intercity distance between cities $i,j$ (with appropriate error exceptions if $i$ or $j$ are not in the range $\{1, \ldots, q_{\dim}(I)\}$).
\item For MILP, the typical oracle uses the following queries: the dimension queries for $n$ and $d$ (unless one or both of them are fixed and known), a query $q^A_{ij}(I)$ that reports the entry of matrix $A$ in row $i$ and column $j$ for the instance $I$, and similar queries $q^B_{ij}, q^b_i, q^c_j$ for the matrix $B$, and vectors $b, c$ (with appropriate error exceptions if the queried index is out of bounds).
\item For Nonlinear Optimization, the  most commonly used oracles return function values, gradient/subgradient values, Hessian or higher order derivative values at a queried point. Thus, we have queries such as $q^{f,x}_0(I)$ which returns $f(x)$ for the objective function $f$ in an instance of~\eqref{eq:nonlinear-opt} where $x$ is a point in the appropriate domain of $f$, or the query $q^{f,x}_1(I)$ which returns the gradient $\nabla f(x)$. Similarly, one has queries for the constraints. Often the set version of these oracles are assumed instead, specially in convex optimization, where one is given a separation oracle for the feasible region and the epigraph of the objective function; see~\cite{lovasz1986algorithmic,GroetschelLovaszSchrijver-Book88} for a well developed theory and applications in this setting.

If the problem class $\I$ has an algebraic form, e.g., polynomial optimization, then the oracle queries may be set up to return the values of the coefficients appearing in the polynomial.
\end{enumerate}
\end{example}

One can seamlessly accommodate oracles with error in the above set-up. For example, the weak separation oracles in~\cite{GroetschelLovaszSchrijver-Book88} can be modeled with no change in the definitions just like strong/exact separation oracles. We will only work with deterministic oracles in this paper. See~\cite{braun2017lower,Nemirovski_Yudin_book} for a discussion of stochastic oracles in the context of continuous convex optimization. We next recall the notion of an oracle Turing machine. 
\begin{definition} An {\em oracle Turing machine} with access to an oracle $(\cQ, H)$ is a Turing machine that has the enhanced ability to pose any query $q\in \cQ$ and use the answer in $H$ in its computations. The queries it poses may depend on its internal states and computations, i.e., it can query adaptively during its processing. 
\end{definition}

We now have everything in place to define what we mean by an optimization algorithm/procedure. Since we focus on the (oracle) Turing machine model and any algorithm will process elements of the solution set $G$ and set of answers $H$ from an oracle, we assume that the elements of these sets can be represented by binary strings or appropriate real number oracles (see Definition~\ref{def:real-oracles}). If one wishes to adopt a different model of computation, then these sets will need representations within that computing paradigm. This clearly affects what class of problems and oracles are permissible. We will not delve into these subtle questions; rather we will stick to our Turing machine model and assume that the class of problems we are dealing with has appropriate representations for $G$ and $H$.

\begin{definition}\label{def:comp} Let $(\I, G, S)$ be an optimization problem class and let $(\cQ, H)$ be an oracle for $\I$. For any $\epsilon \geq 0$, an {\em $\epsilon$-approximation algorithm for $(\I, G, S)$ using $(\cQ, H)$} is an oracle Turing machine with access to $(\cQ, H)$ that starts its computations with the empty string as input and, for any $I\in \I$, ends its computation with an element of $S(I, \epsilon)$, when it receives the answer $q(I)$ for any query $q\in \cQ$ it poses to the oracle. 

If $\cA$ is such an $\epsilon$-approximation algorithm, we define the {\em total complexity} $\comp_{\cA}(I)$ to be the number of elementary operations performed by $\cA$ during its run on an instance $I$ (meaning that it receives $q(I)$ as the answers to any query $q$ it poses to the oracle), where each oracle query counts as an elementary operation (reading the answer of a query may require more than one elementary operation, depending on its length). If $\cA$ makes queries $q_1, \ldots, q_k$ during its run on an instance $I$, we say the {\em information complexity} $\icomp_{\cA}(I)$ is $|q_1(I)| + \ldots + |q_k(I)|$, where $|q_i(I)|$ denotes the length of the binary string or size of the real number oracle (see Definition~\ref{def:real-oracles}) representing the answer $q_i(I)$.

Following standard conventions, the {\em (worst case) complexity} of the algorithm $\cA$ for the problem class $(\I, G, S)$ is defined as $$\comp\textstyle{_{\cA}}:= \sup_{I \in \I}\;\;\comp\textstyle{_{\cA}}(I),$$ and the {\em (worst case) information complexity} is defined as $$\icomp\textstyle{_{\cA}}:= \sup_{I \in \I}\;\;\icomp\textstyle{_{\cA}}(I),$$
\end{definition}

\begin{remark}\label{rem:lower bound} One can assume that any algorithm that poses a query $q$ reads the entire answer $q(I)$ for an instance $I$. Indeed, if it ignores some part of the answer, then one can instead consider the queries that simply probe the corresponding bits that are used by the algorithm (and we may assume our oracles to have this completeness property). We will assume this to be the case in the rest of the paper. Such an assumption can also be made without loss of generality in any other reasonable model of computation.

This implies that the information complexity $\icomp_{\cA}(I)$ is less than or equal to the total complexity $\comp_{\cA}(I)$ of any algorithm $\cA$ running on any instance $I$.

An important notion of complexity that we will not discuss in depth in this paper is that of {\em space complexity}. This is defined as the maximum amount of information (from the oracle queries) and auxiliary computational memory that is maintained by the algorithm (oracle Turing machine) during its entire run. As in Definition~\ref{def:comp}, one can define the {\em total space complexity} and the {\em information space complexity}. Both notions can be quite different from $\comp\textstyle{_{\cA}}$ and $\icomp\textstyle{_{\cA}}$ respectively, since one keeps track of only the amount of information and auxiliary data held in memory at any given stage of the computation, as opposed to the overall amount of information or auxiliary memory used. In many optimization algorithms, it is not necessary to maintain all of the answers to previous queries or computations in memory. A classic example of this is the (sub)gradient descent algorithm where only the current function values and (sub)gradients are stored in memory for computations, and they are not needed in subsequent iterations.
\end{remark}

\section{Oracle ambiguity and lower bounds on complexity}\label{sec:ambiguity} For many settings, especially problems in numerical optimization, a finite number of oracle queries may not pin down the exact problem one is facing. For example, consider $(\I,G,S)$ to be the problem class of the form~\eqref{eq:nonlinear-opt} where $f, g_1, \ldots, g_m$ can be any convex, continuously differentiable functions, and suppose the oracle $(\cQ,H)$ allows function evaluation and gradient queries. Given any finite number of queries, there are infinitely many instances that give the same answers to those queries. 

\begin{definition} Let $(\I,G,S)$ be an optimization problem class and let $(\cQ,H)$ be an oracle for $\I$. For any subset $Q \subseteq \cQ$ of queries, define an equivalence relation on $\I$ as follows: $I \sim_Q I'$ if $q(I) = q(I')$ for all $q\in Q$. For any instance $I$, let $V(I, Q)$ denote the equivalence class that $I$ falls in, i.e., $$V(I, Q)= \{I': q(I) = q(I')\;\; \forall q \in Q\}.$$
\end{definition}

The above definition formalizes the fact that if one only knows the answers to queries in $Q$, then one has a course-grained view of $\I$. This is why the notion of an $\epsilon$-approximate solution becomes especially pertinent. The following theorem gives a necessary condition on the nature of queries used by any $\epsilon$-approximation algorithm.

\begin{theorem}\label{thm:eps-condition} Let $(\I, G, S)$ be an optimization problem class and let $(\cQ, H)$ be an oracle for $\I$. If $\cA$ is an $\epsilon$-approximation algorithm for this optimization problem for some $\epsilon\geq 0$, then

$$\bigcap_{I' \in V(I, Q(I))} S(I',\epsilon) \neq \emptyset\qquad \forall I \in \I,$$ where $Q(I)$ is the set of queries used by $\cA$ when processing instance $I$.
\end{theorem} 

\begin{proof} If $\cA$ is a correct $\epsilon$-approximation algorithm, then suppose it returns $x \in S(I,\epsilon)$ for instance $I$. Since the answers to its queries are the same for all $I'\in V(I,Q(I))$, it will also return $x$ when the answers it receives are $q(I')$ for $q \in Q(I)$ and $I'\in V(I,Q(I))$. Therefore, $x \in S(I',\epsilon)$ for all $I'\in V(I,Q(I))$.
\end{proof}

This leads us to the following definition.

\begin{definition} Let $(\I, G, S)$ be an optimization problem class and let $(\cQ, H)$ be an oracle for $\I$. Let $2^{(\cQ\times H)}$ denote the collection of all {\em finite} sets of pairs $(q,h) \in \cQ\times H$. 

An {\em adaptive query strategy} is a function $D: 2^{(\cQ\times H)} \to \cQ$. The {\em transcript $\Pi(D,I)$ of a strategy $D$ on an instance $I$} is the sequence of query and response pairs $(q_i, q_i(I))$, $i=1,2,\ldots$ obtained when one applies $D$ on $I$, i.e., $q_1 = D(\emptyset)$ and $q_i = D(\{(q_1, q_1(I)), \ldots, (q_{i-1},q_{i-1}(I))\})$ for $i \geq 2$. $\Pi_k(D,I)$ will denote the truncation of $\Pi(D,I)$ to the first $k$ terms, $k \in \N$. We will use $Q(D,I)$ (and $Q_k(D,I)$) to denote the set of queries in the transcript $\Pi(D,I)$ (and $\Pi_k(D,I)$). Similarly, $R(D,I)$ (and $R_k(D,I)$) will denote the set of responses in the transcript.

\noindent The {\em $\epsilon$-information complexity of an instance $I$ for an adaptive strategy $D$} is defined as 
$$\begin{array}{rcl}\icomp\textstyle{_\epsilon}(D,I) &:= &\inf\left\{\sum_{r\in R_k(D,I)} |r|: k \in \N \;\; \textrm{such that}\;\; \bigcap_{I' \in V(I, Q_k(D,I))} S(I',\epsilon) \neq \emptyset\right\}\\
&= &\inf\left\{\sum_{r\in R_k(D,I)} |r|: k \in \N \;\; \textrm{such that}\;\; \bigcap_{I': \Pi_k(D,I') = \Pi_k(D,I)} S(I',\epsilon) \neq \emptyset\right\}
\end{array}$$
The {\em $\epsilon$-information complexity of an adaptive strategy $D$ for the problem class $(\I, G, S)$} is defined as $$\icomp\textstyle{_\epsilon}(D) := \sup_{I\in \I}\;\;\icomp\textstyle{_\epsilon}(D,I)$$
The {\em $\epsilon$-information complexity of the problem class $(\I, G, S)$ with respect to oracle $(\cQ, H)$} is defined as
$$\icomp\textstyle{_\epsilon}:=\inf_D \;\; \icomp\textstyle{_\epsilon}(D),$$ where the infimum is over all possible adaptive queries.
\end{definition}

\begin{remark} In the definitions above, one could restrict adaptive query strategies to be computable, or even polynomial time computable (in the size of the previous query-response pairs). We are not aware of any existing research where such restrictions have been studied to get a more refined analysis of $\epsilon$-information complexity. The typical lower bounding techniques directly lower bound $\icomp_\epsilon$ defined above. One advantage of this is that one does not have to rely on any complexity theory assumptions such as $P \neq NP$ and the lower bounds are unconditional.\end{remark}

We can now formally state the results for lower bounding algorithmic complexity. Remark~\ref{rem:lower bound} and Theorem~\ref{thm:eps-condition} imply the following.

\begin{corollary}\label{cor:lower-bnd} Let $(\I, G, S)$ be an optimization problem class and let $(\cQ, H)$ be an oracle for $\I$. If $\cA$ is an $\epsilon$-approximation algorithm for $(\I, G, S)$ using $(\cQ, H)$ for some $\epsilon \geq 0$, then $$\icomp\textstyle{_\epsilon} \leq \icomp_{\cA} \leq \comp_{\cA}.$$
\end{corollary}

\begin{remark} If $S, S'$ are two different solution operators for $\I, G$ such that $S(I,\epsilon) \subseteq S'(I,\epsilon)$ for all $I \in \I$ and $\epsilon\geq 0$, i.e., the operator $S$ is stricter than $S'$, then the complexity measures with respect to $S$ are at least as large as the corresponding measures with respect to $S'$.
\end{remark}

\begin{remark} In the literature, $\icomp_\epsilon$ is often referred to as the {\em analytical complexity} of the problem class (see, e.g.,~\cite{Nesterov-Book04}). We prefer the phrase {\em information complexity} since we wish to have a unified framework for continuous and discrete optimization and ``analytical" suggests problems that are numerical in nature or involve the continuum. Another term that is used in the literature is {\em oracle complexity}. This is better, in our opinion, but still has the possibility to suggest the complexity of implementing the oracle, rather than the complexity of the queries. Since $\icomp_\epsilon$ is very much inspired by information theory ideas, we follow the trend~\cite{traub1983information,nemirovski1994efficient,braun2017lower,carmon,yudin1976informational} of using the term {\em information complexity (with respect to an oracle)}.

$\comp_{\cA}$ is sometimes referred to as {\em arithmetic complexity}~\cite{Nesterov-Book04} or {\em combinatorial complexity}~\cite{traub1983information} of $\cA$. We prefer to stick to the more standard terminology of simply {\em (worst case) complexity} of the algorithm $\cA$.
\end{remark}

\section{What is the size of an optimization problem?} 
\subsection{Size hierarchies} The notions of complexity defined so far are either too fine or too course. At one extreme is the instance dependent notions $\comp_{\cA}(I)$, $\icomp_{\cA}(I)$ and $\icomp_\epsilon(D,I)$, and at the other extreme are the worst case notions $\comp_{\cA}$, $\icomp_{\cA}$ and $\icomp_\epsilon$. It is almost always impossible give a fine tuned analysis of the instance based complexity notions; on the other hand, the worst case notions give too little information, at best as a function of $\epsilon$, and in the worst case, these values are actually $\infty$ for most problem classes of interest. Typically, a middle path is taken where a countable hierarchy of the problem class is defined and the complexity is analyzed as a function of the levels in the hierarchy.

\begin{definition}\label{def:size} Let $(\I, G, S)$ be an optimization problem class. A {\em size hierarchy} is a countable, increasing sequence $\I_1 \subseteq \I_2 \subseteq \I_3 \subseteq \ldots$ of subsets of $\I$ such that $\I = \bigcup_{k\in \N} \I_k$. The {\em size of any instance $I$ with respect to a size hierarchy} is the smallest $k\in \N$ such that $I \in \I_k$.

The (worst case) complexity of any algorithm $\cA$ for the problem, with respect to the size hierarchy, is defined naturally as $$\comp\textstyle{_{\cA}}(k) := \sup_{I\in \I_k} \comp\textstyle{_{\cA}}(I).$$

Similarly, the (worst case) information complexity of any algorithm $\cA$ for the problem, with respect to the size hierarchy, is defined as $$\icomp\textstyle{_{\cA}}(k) := \sup_{I\in \I_k} \icomp\textstyle{_{\cA}}(I),$$ and the (worst case) $\epsilon$-information complexity of the problem class, with respect to the size hierarchy, is defined as $$\icomp\textstyle{_\epsilon}(k) := \inf_D\;\;\sup_{I\in \I_k}\;\;\icomp\textstyle{_\epsilon}(D,I),$$ where the infimum is taken over all adaptive strategies $D$.
\end{definition}

\begin{example}\label{ex:concrete-size} We review the standard size hierarchies for the problems considered in Example~\ref{ex:concrete-ex}.
\begin{enumerate}
\item In the TSP problem class defined in Example~\ref{ex:concrete-ex}, the standard ``binary encoding" size hierarchy defines $\I_k$ to be all instances such that $\sum_{i,j=1}^n \lceil\log(d_{ij})\rceil \leq k$ (so $k$ must be at least $n^2$), where $d_{ij}\in \Z_+$ are the intercity distances. If one works with real numbers as distances, the size is defined using the sizes of the real number oracles (see Definition~\ref{def:real-oracles}). If one focuses on the so-called {\em Euclidean TSP} instances, then one can define a different size hierarchy based on the number of bits needed to encode the coordinates of the cities (or sizes of the real number oracles).

Another alternative is to simply define $\I_k$ to be all instances with at most $k$ cities.
\item For MILPs, the standard ``binary encoding" size hierarchy defines $\I_k$ to be all instances such that the total number of bits (or sizes of real number oracles) needed to encode all the entries of $A,B,b, c$ is at most $k$. Another alternative is to simply define $\I_k$ as those instances where $m(n+d) \leq k$, or even simply those instances with $n+d \leq k$.

\item For nonlinear optimization problems of the form~\eqref{eq:nonlinear-opt}, often the notion of ``binary encoding" is not meaningful. Consider, for example, the problem of minimizing a linear function $f(x) = c^Tx$ over a full-dimensional, compact convex body $C$ given via a separation oracle. A size hierarchy that has been commonly used in this setting defines $\I_k$ as follows: An instance $I \in \I_k$ if there exist rational numbers $R, r$ such that $C$ is contained in the ball of radius $R$ around the origin, $C$ also contains a ball of radius $r$ inside it (the center may not be the origin), and the total number of bits needed to encode $R$, $r$ and the coordinates of $c$ is at most $k$. See~\cite{lovasz1986algorithmic,GroetschelLovaszSchrijver-Book88} for a fuller discussion and other variants. 
\end{enumerate}
\end{example}

The idea of a size hierarchy is meant to formalize the notion that problems with larger size are ``harder" to solve in the sense that it should take an algorithm longer to solve them. This obviously is often a subjective matter, and as discussed in the above examples, different size hierarchies may be defined for the same optimization problem class. The complexity measures as a function of size is very much dependent on this choice (and also on the model of computation because different models measure complexity in ways that are different from the oracle Turing machine model).

Even within the Turing machine model of computation, a classical example of where different size hierarchies are considered is the optimization problem class of {\em knapsack problems}. Here, one is given $n$ items with weights $w_1, \ldots, w_n \in \Z_+$ and values $v_1, \ldots, v_n \in \Z_+$, and the goal is to find the subset of items with maximum value with total weight bounded by a given budget $W \in \Z$. The standard ``binary encoding" size hierarchy defines the level $\I_k$ to be all problems where $\sum_{i=1}^n (\lceil\log w_i\rceil + \lceil\log v_i\rceil) + \lceil\log W\rceil \leq k$. However, one can also stratify the problems by defining $\I_k$ to be all problems where $\sum_{i=1}^n (\lceil\log w_i\rceil + \lceil\log v_i\rceil) + W \leq k$. The well-known dynamic programming based algorithm has complexity $\comp_{\cA}(k)$ which is exponential in $k$ with respect to the first size hierarchy, while it is polynomial in $k$ with respect to the second size hierarchy. 

The standard ``binary encoding" size hierarchies are the most commonly used ones, motivated by the fact that one needs these many bits to ``write down the problem" for the Turing machine to solve. As we see above, if one develops a unified theory for discrete and continuous optimization based on oracle Turing machines (or other more flexible models of computation), the ``binary encoding" idea loses some of its appeal. And even within the realm of discrete optimization on conventional Turing machines, there is no absolute objective/mathematical principle that dictates the choice of a size hierarchy and, in our opinion, there is subjectivity in this choice. Therefore, complexity measures as a function of the size hierarchy have this subjectivity inherent in them. 

\subsection{More fine-grained parameterizations}

The approach of a size hierarchy as discussed in the previous section parameterizes the instances in a one-dimensional way using the natural numbers. One may choose to stratify instances using more than one parameter and define the complexity measures as functions of these parameters. This is especially useful in continuous, numerical optimization settings where the standard ``binary encoding" is unavailable to define a canonical size hierarchy as in the case of discrete optimization problems. For example, consider the problems of the form~\eqref{eq:nonlinear-opt} where the functions $f, g_1, \ldots, g_m$ are convex, with the solution operator $S(I,\epsilon)$ consisting of those solutions that satisfy the constraints up to $\epsilon$ slack, i.e., $g_i(x) \leq \epsilon$ and have objective value $f(x)$ within $\epsilon$ of the optimal value. We also consider access to a first-order oracle for these functions (see Example~\ref{ex:concrete-oracles}, part 3.). We now create a semi-smooth parameterization of the family of instances using three parameters $d \in \N$ and $R, M \in \R$: $\I_{d, R,M}$ are those instances such that 1) the domains of the functions is $\R^d$, 2) the feasible region $\{x\in \R^d: g_i(x) \leq 0\; i=1, \ldots, m\}$ is contained in the box $\{x \in \R^d: \|x \|_\infty \leq R\}$, and 3) $f, g_1, \ldots, g_m$ are Lipschitz continuous with Lipschitz constant $M$ on this box (a convex function is Lipschitz continuous on any compact set). One can then define the complexity measures $\comp_{\cA}(d,R,M)$ and $\icomp_{\cA}(d,R,M)$ for any $\epsilon$-approximation algorithm $\cA$, and the algorithm independent complexity measure $\icomp_\epsilon(d,R,M)$, as functions of these three parameters (as well as $\epsilon$, of course). As an example one can show $\icomp_{\epsilon}(d,M,R) \in \Theta(d\log(\frac{MR}{\epsilon}))$; see Section~\ref{sec:information-comp}. 

As in the case of size hierarchies, the goal is to tread a middle path between the two extremes of very fine-grained instance dependent measures, or worst case values over all instances (as a function of $\epsilon$). Parameterizing the problem class with more than one parameter gives a little more information. Several examples of such parameterizations for classes of convex optimization problems is presented in~\cite{Nemirovski_Yudin_book}. Our discussion in the next part will involve similar parameterizations of mixed-integer optimization problems. See also the related area of computational complexity theory of {\em parameterized complexity} and {\em fixed-parameter tractability (FPT)}~\cite{downey2012parameterized}.

\part{\Large Complexity of mixed-integer convex optimization}

After setting up the framework in Part I, we now derive concrete results for the class of mixed-integer convex optimization problems. More precisely, we will consider problems of the form

\begin{equation}\label{eq:MICO}\inf\{f(x,y): (x,y) \in C, (x,y) \in \Z^n \times \R^d\}.\end{equation} where $f:\R^n\times \R^d \to \R$ is a convex (possibly nonsmooth) function and $C\subseteq \R^n\times \R^d$ is a closed, convex set, i.e., an instance $I$ is given by $f,C$. We will consider the solution operator $S(I,\epsilon)$ to be all feasible solutions in $C$ that have value at most $\epsilon$ more than the optimal value. One could allow solutions within $\epsilon$ distance of $C$ and all of the results given below can be modified accordingly, but we will consider only truly feasible solutions. 

The following definition will be useful in what follows.

\begin{definition} A {\em fiber box} in $\Z^n\times \R^d$ is a set of the form $\{x\}\times [\ell_1, u_1]\times \ldots [\ell_d, u_d]$ where $x \in \Z^n$ and $\ell_i, u_i \in \R$ for $i=1, \ldots, d$. The {\em length} of the box in coordinate $j$ is $u_j - \ell_j$. The {\em width} of such a fiber box is the minimum of $u_j - \ell_j$, $j=1, \ldots, d$. If $n=0$, a fiber box is simply a hypercuboid in $\R^d$. A fiber box is the empty set if $u_i < \ell_i$ for some $i=1, \ldots, d$. If all $\ell_i = -\infty$ and all $u_i = \infty$, then the set is simply called a {\em fiber} over $x$.
\end{definition}

\section{$\epsilon$-information complexity}\label{sec:information-comp}

In this section, we establish the best-known lower and upper bounds on the $\epsilon$-information complexity of~\eqref{eq:MICO} in the literature. To get the tightest bounds, we will restrict our attention to problems with bounded feasible regions and therefore a minimum solution exists. Moreover, we will also focus on ``strictly feasible" instances. 

\begin{definition}\label{def:instance-class} We parameterize the instances using five parameters $n, d \in \N$ and $R, M, \rho \in \R$. $\I_{n, d, R,M, \rho}$ are those instances such that

\begin{enumerate}\item The domain of $f$ and $C$ are both subsets of $\R^{n} \times \R^{d}$.
\item $C$ is contained in the box $\{z \in \R^n \times \R^d: \|z \|_\infty \leq R\}$, and
\item $f$ is Lipschitz continuous with respect to the $\| \cdot \|_\infty$-norm with Lipschitz constant $M$ on any fiber box of the form $\{x\} \times [-R,R]^d$ with $x \in [-R, R]^n \cap \Z^n$, i.e., for any $(x, y), (x, y')$ with $\|y - y'\|_\infty \leq R$, $|f(x,y) - f(x,y')|\leq M\|y - y'\|_\infty$.
\item If $(x^\star, y^\star)$ is the optimum solution, then there exists $\hat y \in \R^d$ and $0 < \rho \leq 1$ such that $\{(x^\star,y): \|y - \hat y\|_\infty \leq \rho\} \subseteq C$, i.e., there is a ``strictly feasible" point $(x^\star, \hat y)$ in the same fiber as the optimum $(x^\star, y^\star)$ with a fiber box of width $\rho$ in $\R^d$ (the continuous space) around $(x^\star, \hat y)$ contained in $C$. Note that if $d=0$ (the {\em pure integer} case), then this requirement becomes vacuous; consequently, the bounds below in Theorem~\ref{thm:information-comp} for the pure integer case do not involve $\rho$. Also, the assumption $\rho \leq 1$ is not restrictive in the sense that if the condition is satisfied for some $\rho > 0$, then it is also satisfied for $\min\{\rho, 1\}$. Thus, one could alternatively leave this condition out, and the stated bounds below will be modified by replacing $\rho$ with $\min\{\rho, 1\}$.
\end{enumerate}
Table~\ref{tab:parameters} gives a synopsis.
\end{definition}

\begin{table}[htbp]
\begin{center}
\begin{tabular}{|c|l|}
\hline
Parameter & Meaning \\
\hline
& \\
$n$ & number of integer variables \\
& \\
$d$ & number of continuous variables  \\
& \\
$R$ & Boundedness parameter for the feasible region \\
& \\
$\rho$ & Strict feasibility parameter for the feasible region\\
& \\
$M$ & Lipschitz constant for the objective function \\
& \\
\hline
\end{tabular}
\end{center}
\caption{Parameters of the problem instance used to state the complexity bounds}\label{tab:parameters}
\end{table}

In the statement of the result, we will ignore the sizes of the subgradients, function values and separating hyperplanes reported in the answers to oracle queries (which is technically included in our definition of $\icomp_\epsilon$). Thus, we will give lower and upper bounds on the {\em number} of oracle queries only. Taking the sizes of the subgradients and real numbers involved in the answers leads to several interesting questions which, to the best of our knowledge, have not been fully worked out in detail. To keep the discussion aligned with the focus in the literature, we leave these subtleties out of this presentation. This obviously has implications for space information complexity as well.

The bounds below for the mixed-integer case $n, d \geq 1$ are minor adaptations of arguments that first appeared in~\cite{oertel2014integer,basu2017centerpoints}. The main difference is that our presentation here uses the more general information theoretic language developed in Part I, whereas the results in~\cite{oertel2014integer,basu2017centerpoints} were stated for a certain class of algorithms called {\em cutting plane algorithms} (see Section~\ref{sec:branch-and-cut}).

\begin{theorem}\label{thm:information-comp} Let the oracle access to an instance $f, C$ of for~\eqref{eq:MICO} in $\I_{n, d, R,M, \rho}$ from Definition~\ref{def:instance-class} be through a separation oracle for $C$, and a first-order oracle for $f$, i.e., one can query the function value and the subdifferential for $f$ at any point. As a quick legend: $n$

\paragraph{Lower bounds}
\begin{itemize}
\item If $n,d\geq 1$, $$\icomp\textstyle{_\epsilon}(n,d,R,M,\rho) \in \Omega\left(d2^n \log\left(\frac{R}{\rho}\right)\right).$$
\item If $d=0$, $$\icomp\textstyle{_\epsilon}(n,d,R,M,\rho) \in \Omega\left(2^n \log\left(R\right)\right).$$
\item If $n=0$, $$\icomp\textstyle{_\epsilon}(n,d,R,M,\rho) \in \Omega\left(d \log\left(\frac{MR}{\rho\epsilon}\right)\right).$$

\end{itemize}
\paragraph{Upper bounds}
\begin{itemize}
\item If $n, d \geq 1$
$$\icomp\textstyle{_\epsilon}(n,d,R,M,\rho) \in O\left((n+d)d2^n  \log\left(\frac{MR}{\rho\epsilon}\right)\right).$$
\item If $d=0$
$$\icomp\textstyle{_\epsilon}(n,d,R,M,\rho) \in O\left(n2^n  \log(R)\right).$$
\item If $n=0$
$$\icomp\textstyle{_\epsilon}(n,d,R,M,\rho) \in O\left(d  \log\left(\frac{MR}{\rho\epsilon}\right)\right).$$
\end{itemize}
\end{theorem}
\bigskip
\bigskip

Note that when $n=0$, i.e., we consider continuous convex optimization with no integer variables, we have $\icomp\textstyle{_\epsilon}(n,d,R,M,\rho) = \Theta\left(d  \log\left(\frac{MR}{\rho\epsilon}\right)\right)$, giving a tight characterization of the complexity. In fact, these results can be obtained for a much broader class of oracles that include first-order/separation oracles as special cases; see~\cite{Nemirovski_Yudin_book,Nesterov-Book04,nemirovski1994efficient,braun2017lower}.

For pure integer optimization with $d=0$, our upper and lower bounds are off by a linear factor in the dimension, which is of much lower order compared to the dominating term of $2^n\log(R)$. Put another way, both bounds are $2^{O(n)}\log(R)$. The lower bounds come from the feasibility question (see the proofs below). Additionally, since the strict feasibility assumption is vacuous and for small enough $\epsilon > 0$, $S(I,\epsilon)$ is the set of exact optimum solutions, $M,\epsilon$ and $\rho$ do not play a role in the upper and lower bounds; in particular, they are the bounds for obtaining exact solutions ($\epsilon = 0$) as well. 

There seems to be scope for nontrivial improvement in the bounds presented for the mixed-integer case, i.e., $n, d \geq 1$:

\begin{enumerate}
\item It would be nice to unify the lower bound for $n=0$ (the continuous case) and $n\geq 1$ (the truly mixed-integer case). The proof below for $n, d\geq 1$ is based on the feasibility question, which is why $M$ and $\epsilon$ do not appear in the lower bound. This is inspired by the proof technique in~\cite{basu2017centerpoints}. We do not see a similar way to incorporate the objective function parameters to match the upper bound. We suspect that one should be able to prove the stronger lower bound of $\Omega\left(d2^n \log\left(\frac{MR}{\rho\epsilon}\right)\right)$, but at present we do not see how to do this and we are not aware of any existing literature that achieves this. 
\item When one plugs in $n=0$ in the mixed-integer upper bound ($n, d \geq 1$), one does not recover the tight upper bound for $n=0$; instead, the bound is off by a factor of $d$. We believe this can likely be improved, for example, if Conjecture~\ref{conj:mixed-center} below is proved to be true in the future. Then one would have an upper bound of $O\left((n+d)2^n\log\left(\frac{MR}{\rho\epsilon}\right)\right)$ in the mixed-integer case that more accurately generalizes both the pure continuous ($n=0$) and pure integer ($d=0$) upper bounds.
\end{enumerate}

\subsection{Proof of the lower bounds in Theorem~\ref{thm:information-comp}} The general strategy is the following: Given any adaptive query sequence $D$, we will construct two instances $(f_1, C_1), (f_2, C_2) \in \I_{n, d, R,M, \rho}$ such that the transcripts $\Pi_k(D,(f_1, C_1))$ and $\Pi_k(D,(f_2, C_2))$ are equal for any $k$ less than the lower bound, but $S((f_1, C_1), \epsilon) \cap S((f_2, C_2),\epsilon) =\emptyset$.

\paragraph{The mixed-integer case ($n,d\geq 1$).} We will show that $\icomp\textstyle{_\epsilon}(n,d,R,M,\rho) \geq d2^n \log_2\left(\frac{R}{3\rho}\right).$ We construct $C_1, C_2 \subseteq \R^{n}\times \R^d$ such that $C_1 \cap C_2 \cap (\Z^n \times \R^d)= \emptyset$,  both sets satisfy the strict feasibility condition dictated by $\rho$, and any separation oracle query from $D$ on $C_1$ and $C_2$ has the same answer. Our instances will consist of these two sets as feasible regions and $f_1 = f_2$ as constant functions, thus any first-order oracle query in $D$ will simply return this constant value and $0$ as a subgradient. Since there is no common feasible point, $S((f_1, C_1), \epsilon) \cap S((f_2, C_2),\epsilon) =\emptyset$ as required.

The construction of $C_1$ and $C_2$ goes as follows. Since the function oracle calls are superfluous, we may assume the $k$ (adaptive) queries $\{q_1, \ldots, q_k\}$ to be all separation oracle queries. Begin with $X_0 = [0,1]^n \times [0,R]^d$. We create a nested sequence $X_0 \supseteq X_1 \supseteq X_2 \supseteq \ldots \supseteq X_k$ such that $X_i \cap (\{x\}\times \R^d)$ is a fiber box (possibly empty) for any $x\in \{0,1\}^n$. $X_i$ is defined inductively from $X_{i-1}$, using the query $q_i$. For every $\tilde x \in \{0,1\}^n$, we maintain a counter $\#\tilde x(i)$ which will keep track of how many $q_j$, $j\leq i$ queried a point of the form $(\tilde x, y)$ inside $X_{j-1}$ for some $y \in \R^d$.

If $q_i$ queries $(x^i,y^i) \not\in X_{i-1}$, then we simply report any hyperplane separating $(x^i,y^i)$ from $X_{i-1}$ as the answer to $q_i$ and define $X_i = X_{i-1}$. If $q_i$ queries $(x^i,y^i) \in X_{i-1}\setminus (\Z^n \times \R^d)$ (i.e., $x^i \not\in \Z^n)$, we define $X_i = X_{i-1}$ and the answer to the query $q_i$ is that $(x^i,y^i)$ is in the set. 

Suppose now $(x^i,y^i) \in X_{i-1}\cap (\Z^n \times \R^d)$. If $\#x^i(i-1) \geq d \log_2\left(\frac{R}{3\rho}\right)$ then we report a halfspace $H$ that separates $X_{i-1} \cap (\{x_i\}\times \R^d)$ from the rest of the fibers $X_{i-1}\cap(\{x\}\times \R^d)$ for $x \neq x_i$, and define $X_i = X_{i-1}\cap H$. If $\#x^i(i-1) < d \log_2\left(\frac{R}{3\rho}\right)$ then select the coordinate $j = (\#x^i(i-1) \mod d) + 1$ and define the hyperplane $\{(x,y) \in \R^n\times \R^d: y_j = y^i_j\}$. Let $B$ denote the fiber box $X_{i-1}\cap (\{x^i\}\times \R^d)$. Consider the separation with a halfspace $\hat H$ with this hyperplane such that $B \cap \hat H$ has length in coordinate $j$ to be at least half of the length $B$ in coordinate $j$. We now rotate this hyperplane and halfspace $\hat H$ to obtain a halfspace $H$ such that $X_{i-1} \cap H$ has the same intersection as $X_{i-1}$ with $(\{x\}\times \R^d)$ for $x \neq x_i$. In other words, all other mixed-integer fibers in $X_{i-1}$ are maintained. Define $X_i = X_{i-1} \cap H$ and $H$ as the separating halfspace for query $q_i$. Update $\#x^i(i) = \#x^i(i-1) + 1$. Note that the above construction ensures inductively that for any $i\in \{1, \ldots, k\}$, the set $X_i \cap (\{x\}\times \R^d)$ is a fiber box for $x \in \{0,1\}^n$.

Since $\sum_{x\in \{0,1\}^n} \#x(k) \leq k < 2^n\cdot d \log_2\left(\frac{R}{3\rho}\right)$, we observe that $X_k$ contains a fiber box $B$ of width at least $3\rho$. Thus, we can select two fiber boxes $B_1, B_2 \subseteq B$ such that $B_1 \cap B_2 = \emptyset$, and $B_1$ and $B_2$ have width $\rho$. For $i=1,2$, define $C_i$ to be the convex hull of $B_i$ and all the points queried by $D$ that were reported to be in the set. We observe that $C_i\cap (\Z^n\times \R^d) = B_i$ for $i=1,2$ and thus we have no common feasible points in $C_1, C_2$. This completes the proof for $d\geq 1$.

\paragraph{The pure integer case ($d=0$).} The proof proceeds in a similar manner to the mixed-integer case ($n, d \geq 1)$ with $X_0 = [0,1]^{n-1} \times [0,\lfloor R\rfloor] \subseteq \R^n$. The ``fibers'' are now $\{x\} \times \{0,1,\ldots, \lfloor R\rfloor\}$. If $k < 2^n \log_2(R)$, one can again construct $C_1, C_2 \subseteq X_0$ such that $C_1 \cap C_2 \cap \Z^n = \emptyset$ by an inductive argument based on the queries from $D$, and take $f_1, f_2$ as constant functions.

\paragraph{The pure continuous case ($n=0$).} We omit the proof as this has appeared in many different places in the literature~\cite{Nemirovski_Yudin_book,Nesterov-Book04,nemirovski1994efficient,braun2017lower}. The idea is very similar to what was presented above for the general mixed-integer case. One proves that $$\begin{array}{rcl}\icomp\textstyle{_\epsilon}(d,R,M,\rho) & \geq &\max\left\{d \log_2\left(\frac{R}{3\rho}\right),d \log_{2}\left(\frac{MR}{8\epsilon}\right)\right\} \\ & \geq & \frac{d \log_2\left(\frac{R}{3\rho}\right) + d \log_{2}\left(\frac{MR}{8\epsilon}\right)}{2} \\ & \in & \Omega\left(d \log\left(\frac{MR}{\rho\epsilon}\right)\right).\end{array}$$
If $k < d \log_2\left(\frac{R}{3\rho}\right)$ one can appeal to the mixed-integer case above. In fact, there is no rotation of halfspaces necessary as there are no integer fibers. If $k < d \log_{2}\left(\frac{MR}{8\epsilon}\right)$, one constructs two different convex functions $f_1, f_2$ while the feasible region can be taken to be $[0,R]^d$ in both cases. The details are a little more complicated than the separation oracle case, since the function values and the subgradients have to be more carefully engineered. We refer the reader to the references cited above for the details.\qed

\subsection{Proof of the upper bounds in Theorem~\ref{thm:information-comp}}

The idea of the upper bound hinges on a geometric concept that has appeared in several different areas of mathematics, including convex geometry, statistics and theoretical computer science.

\begin{definition} For any $S \subseteq \Z^n \times \R^d$ with $d\geq 1$, $\nu(S)$ will denote the {\em mixed-integer volume} of $S$, i.e., $$\nu(S):= \sum_{x \in \Z^n}\mu_d(S \cap (\{x\}\times \R^d)),$$ where $\mu_d$ is the standard Lebesgue measure (volume) in $\R^d$. If $d=0$, we overload notation and use $\nu(S)$ to denote the number of integer points in $S$, i.e., the counting measure on $\Z^n$.
\end{definition}

Note that if $S = C \cap (\Z^n \times \R^d)$ for a compact convex set $C \subseteq \R^n \times \R^d$, then $\nu(S)$ is finite.

\begin{definition} For any $S \subseteq \Z^n \times \R^d$ and $x \in \R^n \times \R^d$, define $$h_S(x):=\inf_{\begin{array}{c}\textrm{halfspace }H: \\ x \in H\end{array}}\;\;\nu(S \cap H).$$

The set of {\em centerpoints of $S$} is defined as $\C(S):= \argmax_{x \in S}h_S(x)$. 
\end{definition}

The above concept was first defined in Timm Oertel's Ph.D. thesis~\cite{oertel2014integer} and extensions were introduced in~\cite{basu2017centerpoints}. We refer the reader to the thesis and the cited paper, and the references therein for structural properties of $h_S$ and $\C(S)$. For our purposes, we will simply need the following result.

\begin{theorem}\label{thm:centerpoint} Let $C \subseteq \R^n \times \R^d$ be any compact, convex set and let $S = C \cap (\Z^n \times \R^d)$. Then $\C(S)$ is nonempty and $h_S(\hat x) \geq \frac{1}{2^n(d+1)}\nu(S)$ for any centerpoint $\hat x$. If $n=0$, then $h_S(\hat x) \geq \left(\frac{d}{d+1}\right)^d\nu(S) \geq \frac{1}{e}\nu(S)$.
\end{theorem}

The first bound in Theorem~\ref{thm:centerpoint} was first established in~\cite{oertel2014integer} and is a special case of a general result involving Helly numbers~\cite[Theorem 3.3]{basu2017centerpoints}. The second bound ($n= 0$) is due to Gr\"unbaum~\cite{Gruenbaum1960}. There is clearly a gap in the two cases and the following sharper lower bound is conjectured to be true~\cite{oertel2014integer,basu2017centerpoints}; a matching upper bound is given by $S = \{0,1\}^n \times \Delta_d$, where $\Delta_d$ is the standard $d$-dimensional simplex.

\begin{conj}\label{conj:mixed-center} Under the hypothesis of Theorem~\ref{thm:centerpoint}, $h_S(\hat x) \geq \frac{1}{2^n}\left(\frac{d}{d+1}\right)^d\nu(S) \geq \frac{1}{2^n}\frac{1}{e}\nu(S)$ for any $n, d \geq 0$ (both not both $0$) for any centerpoint $\hat x$.
\end{conj}

The final piece we need is the following consequence of a ``strict feasibility" type assumption.

\begin{lemma}\label{lem:eps-sol-vol} Let $1 \leq p \leq \infty$. Let $C\subseteq \R^k$ be a closed, convex set such that $\{z\in \R^k: \|z - a \|_p \leq \rho \} \subseteq C \subseteq \{z\in \R^k: \|z \|_p \leq R\}$, for some $R, \rho \in \R_+$ and $a \in \R^k$. Let $f:\R^k \to \R$ be a convex function that is Lipschitz continuous over $\{z\in \R^k: \|z \|_p \leq R\}$ with respect to the $\|\cdot\|_p$-norm with Lipschitz constant $M$. For any $\epsilon \leq 2MR$ and for any $z^\star \in C$, the set $\{z \in C: f(z) \leq f(z^\star) + \epsilon\}$ contains an $\|\cdot\|_p$ ball of radius $\frac{\epsilon\rho}{2MR}$ with center lying on the line segment between $z^\star$ and $a$.
\end{lemma}

\begin{proof} Since $C \subseteq \{z: \|z \|_p \leq R\}$, we must have $C \subseteq \{z: \|z - z^\star\|_p \leq 2R\}$. By convexity of $C$ and the fact that $\frac{\epsilon}{2MR} \leq 1$, $z^\star + \frac{\epsilon}{2MR}(C - z^\star) \subseteq C$. Hence, $$z^\star + \frac{\epsilon}{2MR}(C - z^\star) \subseteq \{z \in C: \|z - z^\star\|_p \leq \frac{\epsilon}{M}\} \subseteq \{z \in C: f(z) \leq f^\star + \epsilon\},$$ where the second containment follows from the Lipschitz property of $f$. Since $C$ contains an $\|\cdot\|_p$ ball of radius $\rho$ centered at $a$, the set $z^\star + \frac{\epsilon}{2MR}(C - z^\star)$ (i.e., the $\frac{\epsilon}{2MR}$ scaling of $C$ about $z^\star$) must contain a ball of radius $\frac{\epsilon\rho}{2MR}$ centered at a point on the line segment between $z^\star$ and $a$.\end{proof}

We now proceed with the proof of the upper bounds in Theorem~\ref{thm:information-comp}.

\paragraph{The mixed-integer case with $n,d\geq 1$.} We consider the following adaptive search strategy. For any finite subset $T \subseteq \cQ \times H$ (possibly empty), where $\cQ$ is the set of all possible first-order or separation oracle queries in $[-R,R]^{n+d}$ and $H$ is the set of possible responses to such queries, define $D(T)$ as follows. Let $z_1, \ldots, z_q$ be the points queried in $T$ where either a first-order oracle call to a function was made, or a separation oracle call was made that returned a separating hyperplane (i.e., the point is not in the set queried). Let $h_j$ be the subgradient or normal vector to the separating hyperplane returned at $z_j$, $j=1, \ldots, q$. Define $v_{\min}$ to be the minimum function value seen so far ($+\infty$ if no first order query exists in $T$). 

The next query for $D$ will be at the centerpoint $\hat z$ of the set $$\left\{z \in \Z^n \times \R^d : 
\begin{array}{ll}\langle h_i, z - z_i \rangle \leq 0 & i = 1, \ldots, q, \\ 
\|x\|_\infty \leq R &\end{array}\right\}$$

If $T$ is the transcript on an instance $f, C$, then the ``search space" polyhedron $P$ defined by the above inequalities contains $C \cap \{z: f(z) \leq v_{\min}\}$. $D$ now queries the centerpoint of $P \cap (\Z^n \times \R^d)$ so that any separating hyperplane or subgradient inequality can remove a guaranteed fraction of the mixed-integer volume of the current search space. More formally, $D$ first queries the separation oracle for $C$ at $\hat z$. If the separation oracle says $\hat z$ is in $C$, then $D$ queries the first-order oracle for $f$ at $\hat z$.

Consider any instance $I = (f, C) \in \I_{n, d, R,M, \rho}$ and any natural number $$k \geq 2\cdot\left(\log_b\left(\left(\frac{2R+1}{\rho}\right)^{n+d}\right) + \log_b\left(\left(\frac{M(2R+1)}{\epsilon}\right)^{n+d}\right)\right),$$ where $b = \frac{2^n(d+1)}{2^n(d+1)-1}$. We claim that at least one first-order oracle query appears in the transcript $\Pi_k(D,I)$ and $z_{\min}\in S(I',\epsilon)$ for every instance $I'$ such that $\Pi_k(D,I') = \Pi_k(D, I)$, where $z_{\min}$ is a point queried in the transcript $\Pi_k(D, I)$ with the minimum function value amongst all points queried with a first-order query on $f$ in $\Pi_k(D, I)$. In other words, for any instance $I' = (f', C')$ such that $\Pi_k(D,I') = \Pi_k(D, I)$, we have $z_{\min} \in C'$ and $f'(z_{\min}) - OPT \leq \epsilon$ where $OPT$ is the minimum value of $f'$ on $C'$. This will prove the result since $$\begin{array}{rcl}2\cdot\left(\log_b\left(\left(\frac{2R + 1}{\rho}\right)^{n+d}\right) + \log_b\left(\left(\frac{M(2R+1)}{\epsilon}\right)^{n+d}\right)\right) & = & 2(n+d)\log_b\left(\frac{M(2R+1)^2}{\rho\epsilon}\right)\\ & = & 2(n+d)\ln\left(\frac{M(2R+1)^2}{\rho\epsilon}\right)/\ln(b) \\ & \leq &2(n+d)2^n(d+1)\ln\left(\frac{M(2R+1)^2}{\rho\epsilon}\right) \\ & \in & O\left((n+d)d2^n\log\left(\frac{MR}{\rho\epsilon}\right)\right)\end{array}$$

First, let $k'$ be the number of queries in $\Pi_k(D,I)$ that were either first-order oracle queries on $f$ or separation oracle queries on $C$ that returned a separating hyperplane, i.e., we ignore the separation oracle queries on points inside $C$. Observe that $k' \geq k/2 \geq\log_b\left(\left(\frac{2R+1}{\rho}\right)^{n+d}\right) + \log_b\left(\left(\frac{M(2R+1)}{\epsilon}\right)^{n+d}\right)$ since a query on any point in $C$ is immediately followed by a first order query on the same point. Theorem~\ref{thm:centerpoint} implies that each of these $k'$ queries reduces the mixed-integer volume of the current search space by at least $1/b$. Recall that we start with a mixed-integer volume of at most $(2R+1)^{n+d}$ and $C$ contains a fiber box of mixed-integer volume at least $\rho^d \geq \rho^{n+d}$ (since $\rho \leq 1$). Thus, at most $\log_b\left(\left(\frac{2R+1}{\rho}\right)^{n+d}\right)$ queries can be separation oracle queries and we have at least $\log_b\left(\left(\frac{M(2R+1)}{\epsilon}\right)^{n+d}\right)$ first-order queries to $f$ at points inside $C\cap (\Z^n\times \R^d)$. Let $k''$ denote the number of such queries, queried at $z_1, \ldots, z_{k''}$ with responses $h_1, \ldots, h_{k''}$ as the subgradients and $v_1, \ldots, v_{k''}$ as the function values. Let $v_{\min}$ be the minimum of these function values, corresponding to the query point $z_{\min}$. Since $z_{\min}$ is feasible to $C$, if $f', C'$ is any other instance with the same responses to all queries in $\Pi_k(D,I)$, then $z_{\min}$ is feasible to $C'$ as well. In fact, all the points $z_1, \ldots, z_{k''}$ are in $C'$. We now verify that $f'(z_{\min}) \leq OPT + \epsilon$ where $OPT$ is the minimum value of $f'$ on $C'\cap (\Z^n\times\R^d)$ attained at, say $z^\star = (x^\star, y^\star)$.

Let $C'' = C' \cap (\{x^\star\}\times \R^d)$ be the intersection of $C'$ with the fiber containing $z^\star$. Consider the polyhedron $$\tilde P:=\{z: \langle h_j, z - z_j \rangle \leq 0 \;\; j = 1, \ldots, k''\}.$$ 
Since we have been reducing the mixed-integer volume at a rate of $1/b$, $C' \cap \tilde P$ has mixed-integer volume at most $(2R+1)^{n+d}/b^{k'}$ and therefore $C''\cap \tilde P$ has $d$-dimensional volume at most $(2R+1)^{n+d}/b^{k'}$. Since $k' \geq \log_b\left(\left(\frac{2R+1}{\rho}\right)^{n+d}\right) + \log_b\left(\left(\frac{M(2R+1)}{\epsilon}\right)^{n+d}\right)$, we must have $b^{k'} \geq \left(\frac{2R+1}{\rho}\right)^{n+d}\cdot\left(\frac{M(2R+1)}{\epsilon}\right)^{n+d}$. Thus, $C''\cap \tilde P$ has $d$-dimensional volume at most $\left(\frac{\rho\epsilon}{M(2R+1)}\right)^{n+d}$. We may assume $\frac{\epsilon}{2MR} \leq 1$, otherwise any feasible solution is an $\epsilon$ approximate solution, and so is $z_{\min}$. Since $\rho \leq 1$ as well, this means $\frac{\rho\epsilon}{M(2R+1)}\leq 1$. Therefore, $\left(\frac{\rho\epsilon}{M(2R+1)}\right)^{n+d} \leq \left(\frac{\rho\epsilon}{M(2R+1)}\right)^{d} < \left(\frac{\rho\epsilon}{2MR}\right)^{d}$. From Lemma~\ref{lem:eps-sol-vol}, $\{z \in C'' : f'(z) \leq f'(z^\star) + \epsilon\}$ has volume at least $\left(\frac{\rho\epsilon}{2MR}\right)^{d}$. Thus, at least one point $\hat z$ in $\{z \in C'' : f'(z) \leq OPT + \epsilon\}$ must be outside $C'' \cap \tilde P$. Such a point must violate one of the subgradient inequalities defining $\tilde P$, say corresponding to index $\tilde j$. In other words,  $\langle h_{\tilde j}, \hat z - z_{\tilde j}\rangle > 0$. This means $f'(\hat z) \geq f'(v_{\tilde j}) + \langle h_{\tilde j}, \hat z - z_{\tilde j}\rangle > f'(v_{\tilde j})$. Thus, $f'(z_{\min}) \leq f'(v_{\tilde j}) < f'(\hat z) \leq OPT + \epsilon$. 

\paragraph{The pure integer case with $d=0$.} The proof proceeds in a very similar manner except that one can stop when we have at most one integer point left in the polyhedral search space. Thus, we start from the box $[-R, R]^n$ containing $(2R+1)^n$ integer points and end with at most a single integer point, removing at least $\frac{1}{2^n}$ fraction of integer points every time by Theorem~\ref{thm:centerpoint}.

\paragraph{The pure continuous case with $n=0$.} The proof is very similar and the only difference is that we can use the stronger bound on the centerpoints due to Gr\"unbaum from Theorem~\ref{thm:centerpoint}. In other words, $b$ can be taken to be the Euler's constant while mimicking the proof of the $n, d\geq 1$ case above.
\qed

\begin{remark} The upper and lower bounds achieved above are roughly a consequence of the concept of Helly numbers~\cite{hoffman1979binding,Doignon1973,bell1977theorem,scarf1977observation,helly1923mengen,AverkovWeismantel12}. For any subset $S\subseteq \R^k$, we say $K_1, \ldots, K_t$ is a {\em critical family of convex sets} with respect to $S$ (of size $t$) if $K_1 \cap \ldots \cap K_t \cap S = \emptyset$, but for any $i \in \{1, \ldots, t\}$, $\cap_{j\neq i}K_j \cap S \neq \emptyset$. The {\em Helly number of $S$} is the size of the largest critical family with respect to $S$ (possibly $+\infty$). It turns out that the Helly number of $\Z^n\times \R^d \subseteq \R^{n+d}$ is $2^n(d+1)$ and there exists a critical family of halfspaces $H_1, \ldots, H_{2^n(d+1)}$ of this size~\cite{hoffman1979binding,AverkovWeismantel12}. Now consider the family of $2^n(d+1)$ polyhedra $\cap_{j\neq i}H_j$ for $i=1, \ldots, 2^n(d+1)$, along with the polyhedron $\cap _{j=1}^{2^n(d+1)}H_j$. If one makes less than $2^n(d+1)$ separation oracle queries, then every time we can simply report the halfspace $H_j$ that does not contain a mixed-integer query point (such a halfspace exists since $\cap _{j=1}^{2^n(d+1)}H_j \cap (\Z^n \times \R^d)= \emptyset$), and if the query point is not in $\Z^n\times \R^d$, we truthfully report if it is in $\cap _{j=1}^{2^n(d+1)}H_j$ or not. The intersection of these reported halfspaces still contains a point from $\Z^n \times \R^d$ since it is a critical family and we have less than $2^n(d+1)$  queries. Therefore, we are unable to distinguish between the case $\cap _{j=1}^{2^n(d+1)}H_j$ which has no point from $\Z^n\times \R^d$ and the nonempty case. This gives a lower bound of $2^n(d+1)$. As we saw in the proof of the upper bound above, the key result is Theorem~\ref{thm:centerpoint} which is based on Helly numbers again~\cite{Gruenbaum1960,oertel2014integer,basu2017centerpoints}.
\end{remark}

\section{Algorithmic complexity}\label{sec:alg-comp}
   
The upper bounds on $\epsilon$-information complexity presented in Section~\ref{sec:information-comp} do not immediately give upper bounds on algorithmic complexity, unless we can provide an algorithm for computing centerpoints. This is computationally extremely challenging~\cite{basu2017centerpoints,oertel2014integer}. The best known algorithms for mixed-integer convex optimization do not match the $\epsilon$-information complexity bounds presented, even in terms of the informational bound $\icomp_{\cA}$ (see Definition~\ref{def:comp}). We will present an $\epsilon$-approximation algorithm for mixed-integer convex optimization in this section whose information complexity bound is the closest known to the algorithm independent $\epsilon$-information complexity bound from the previous section. In the case of pure continuous optimization, i.e., $n=0$, the algorithm's information complexity is larger than the corresponding $\epsilon$-information complexity bound in Theorem~\ref{thm:information-comp} by a factor that is linear\footnote{There is also a factor of $\log(d)$ which shows up due to technical reasons of using $\|\cdot \|_2$ instead of $\|\cdot \|_\infty$.} in the dimension $d$. In the case of pure integer optimization, i.e., $d=0$, the algorithm's information complexity bound is $2^{O(n\log n)}\log(R)$. Compared to the $\epsilon$-information complexity bound of $2^{O(n)}\log(R)$ from Theorem~\ref{thm:information-comp}, there seems to be a significant gap. It remains a major open question in integer optimization whether the gap between $2^{O(n)}$ and $2^{O(n\log n)}$ can be closed or not by designing a better algorithm.

The overall complexity (including the computational complexity) of the algorithm (see Definition~\ref{def:comp}) will be seen to be a low degree polynomial factor larger than its information complexity.

\subsection{Enumeration and cutting planes}

Algorithms for mixed-integer convex optimization are based on two main ideas. The first one, called {\em branching}, is a way to systematically explore different parts of the feasible region. The second aspect, that of cutting planes, is useful when one is working with a relaxation (superset) of the feasible region and uses separating hyperplanes to remove parts of the relaxation that do not contain feasible points. 

\begin{definition}\label{def:branching-and-cuts} A {\em disjunction} for $\Z^n\times \R^d$ is a union of polyhedra $D= Q_1\cup \ldots \cup Q_k$ such that $\Z^n \times \R^d \subseteq D$.

For any set $X$ in some Euclidean space, a {\em cutting plane} for $X$ is a halfspace $H$ such that $X \subseteq H$. If $X$ is of the form $C\cap (\Z^n \times \R^d)$, then the cutting plane is {\em trivial} if $C \subseteq H$, while it is said to be {\em nontrivial} otherwise.
\end{definition}

The words ``trivial" and ``nontrivial" are used here in a purely technical sense. For a complicated convex set $C$, we may have a simple polyhedral relaxation $R \supseteq C$ such as those used in the proofs of upper bounds in Theorem~\ref{thm:information-comp}, and the separation oracle for $C$ can return {\em trivial} cutting planes that shave off parts of $R$. But if the oracle is difficult to implement, there may be nothing trivial about obtaining such a cutting plane. Our terminology comes from settings where $C$ has a simple description and separating from $C$ is not a big deal; rather, the interesting work is in removing parts of $C$ that do not contain any point from $X = C\cap (\Z^n \times \R^d)$. We hope the reader will indulge us in our purely technical use of the terms {\em trivial} and {\em nontrivial} cutting planes.

\begin{example}\label{ex:Disj-CP-ex}
\begin{enumerate}
\item A well-known example of disjunctions for $\Z^n \times \R^d$ is the family of {\em split disjunctions} that are of the form $\{x\in \R^{n+d} : \langle \pi, x \rangle \leq \pi_0\} \cup \{x\in \R^{n+d} : \langle \pi, x \rangle \geq \pi_0+1\}$, where $\pi \in \Z^n \times \{0\}^d$ and $\pi_0\in \Z$. When the first $n$ coordinates of $\pi$ correspond to a standard unit vector, we get {\em variable disjunctions}, i.e., disjunctions of the form $\{x: x_i \leq \pi_0\} \cup \{x : x_i \geq \pi_0+1\}$, for $i=1, \ldots, n$. Several researchers in this area have also considered the intersection of $t$ different split disjunctions to get a disjunction~\cite{li2008cook,dash2013t,dash2014lattice}; these are known as $t$-branch split disjunctions.
\item As mentioned above, for any convex set $C$ contained in a polyhedron $P$, the separation oracle for $C$ can return trivial cutting planes if a point from $P\setminus C$ is queried. Examples of nontrivial cutting planes for sets of the form $C \cap (\Z^n \times \R^d)$ include {\em Chv\'atal-Gomory cutting planes}~\cite[Chapter 23]{sch} and {\em split cutting planes}~\cite{cook-kannan-schrijver}. These will be discussed in more detail below.
\end{enumerate}
\end{example}

\subsection{The ``Lenstra-style" algorithm}\label{sec:lenstra} Cutting plane based algorithms were designed in continuous convex optimization quite early in the development of the subject~\cite{Nesterov-Book04}. In the 80s, these ideas were combined with techniques from algorithmic geometry of numbers and the idea of branching on split disjunctions to design algorithms for the mixed-integer case as well~\cite{Lenstra83,Kannan87,GroetschelLovaszSchrijver-Book88}. There has been a steady line of work since then with sustained improvements; see~\cite{Dadush12,Heinz05,KhachiyanPorkolab00,hildebrand2013new} for a representative sample. We will present here an algorithm based on these ideas whose complexity is close to the best known algorithmic complexity for the general mixed-integer case.

We first introduce some preliminary concepts and results.

\begin{definition} Given a positive definite matrix $A \in \R^{k \times k}$, the {\em norm defined by $A$} on $\R^k$ is $\|x\|_A := \sqrt{x^TA^{-1}x}$. The unit ball of this norm $E_A:= \{x: x^TA^{-1}x \leq 1\}$ is called an {\em ellipsoid defined by $A$}. The orthonormal eigenvectors of $A$ are called the {\em principal axes.}
\end{definition}

The following result, due to Yudin and Nemirovski~\cite{yudin1976informational}, is a foundational building block for the algorithm.

\begin{theorem}\label{thm:ellipsoid}~\cite[Lemma 3.3.21]{GroetschelLovaszSchrijver-Book88} Let $A \in \R^{k\times k}$ be a positive definite matrix. For any halfspace $H$ and any $0 \leq \beta < \frac{1}{k}$ such that $H$ does not contain $\beta E_A$, there exists another positive definite matrix $A'$ and $c \in \R^k$ such that $E_A \cap H \subseteq c + E_{A'}$ and $$\vol(E_{A'}) \leq e^{-\frac{(1-\beta k)^2}{5k}}\vol(E_A),$$ where $\vol(\cdot)$ denotes the $k$-dimensional volume. Moreover, $c$ and $A'$ can be computed from $A$ by an algorithm with complexity $O(k^2)$.
\end{theorem}

We will also need the following fundamental result in geometry of numbers.

\begin{theorem}\label{thm:flatness}[Khinchine's flatness theorem for ellipsoids]\cite{BanaszczykLitvakPajorSzarek99,banaszczyk1996inequalities,rudelson2000distances} Let $E \subseteq \R^k$ be an ellipsoid and $c \in \R^k$ such that $(c +E) \cap \Z^k = \emptyset$. Then there exists $w \in \Z^k \setminus \{0\}$ such that $$\max_{v \in E} \;\langle w, v\rangle - \min_{v \in E}\;\langle w, v\rangle \leq k.$$
\end{theorem}

The final piece we will need is the following algorithmic breakthrough achieved at the beginning of the previous decade~\cite{MicciancioVoulgaris2010,micciancio2013deterministic,aggarwal2015-SVP,aggarwal2015-CVP}. We state the result in a way that will be most convenient for us.

\begin{theorem}\label{thm:SVP-CVP} Let $A \in \R^{k\times k}$ be a positive definite matrix. Then there exist algorithms with worst case complexity $2^{O(k)}\poly(\size(A))$ that solve the following optimization problems: $$\min_{v \in \Z^k\setminus\{0\}} \|v\|_A\qquad \textrm{(Shortest Vector Problem (SVP))}$$ and for any given vector $c \in \R^k$, $$\min_{v \in \Z^k} \|v-c\|_A\qquad \textrm{(Closest Vector Problem (CVP))}$$
\end{theorem}

We are now ready to describe our algorithm. We begin with a feasibility algorithm before discussing optimization. Given a closed, convex set, the algorithm either correctly computes a mixed-integer point in the convex set, or reports that there is no mixed-integer point ``deep inside" the set. Thus, the algorithm is not an exact feasibility algorithm. Nevertheless, this will suffice to design an $\epsilon$-approximation algorithm for the problem class~\eqref{eq:MICO} parameterized by $n,d,M,R, \rho$, as studied in Section~\ref{sec:information-comp}. However, since we work with ellipsoids, the parameters $R$, $\rho$ and the Lipschitz constant $M$ will all use the $\|\cdot \|_2$ norm instead of the $\|\cdot \|_\infty$ norm as in Section~\ref{sec:information-comp}. Moreover, $M$ is defined with respect to the full space $\R^n\times \R^d$, as opposed to just $\R^d$.

\begin{theorem}\label{thm:feas-alg} Let $R \geq 0$. There exists an algorithm $\cA$, i.e., oracle Turing machine, such that for any closed, convex set $C \subseteq \{z \in \R^n\times \R^d: \|z\|_2 \leq R\}$ equipped with a separation oracle that $\cA$ can access, and any $\delta > 0$, either correctly computes a point in $C\cap (\Z^n \times \R^d)$, or correctly reports that there is no point $z \in C \cap (\Z^n \times \R^d)$ such that the Euclidean ball of radius $\delta$ around $z$ is contained in $C$. 

Moreover, $$\icomp\textstyle{_{\cA}}(n,d,R,\delta) \leq 2^{O(n\log(n + d))}\log\left(\frac{R}{\delta}\right)$$ and $$\comp\textstyle{_{\cA}}(n,d,R,\delta) \leq 2^{O(n\log(n+d))}\poly\left(\log\left(\frac{R(n+d)}{\delta}\right)\right)$$
\end{theorem}

\begin{proof} The algorithm uses recursion on the ``integer dimension" $n$. 

Let $c_0=0$ and $E_0 = \{z: \|z\|_2 \leq R\}$. The algorithm will either iteratively compute $c_i \in \R^n \times \R^d$ and ellipsoid $E_i \subseteq \R^n \times \R^d$ from $c_{i-1}, E_{i-1}$ for $i=1, 2, \ldots$ such that the invariant $C\subseteq c_{i} + E_{i}$ is maintained, or the algorithm will recurse on lower dimensional problems.

If $\vol(E_{i-1})$ less than the volume of a Euclidean ball of radius $\delta$, then we report that there is no point $z \in C \cap (\Z^n \times \R^d)$ such that the Euclidean ball of radius $\delta$ around $z$ is contained in $C$. 

Otherwise, we either compute a ``test point" $(\hat x, \hat y) \in \Z^n \times \R^d$ and generate the new $c_i, E_i$ based on properties of this point (Cases 1 and 2a below), or recurse on lower dimensional subproblems (Case 2b below). 
\medskip

\noindent \underline{\bf Case 1: $n=0$.} Define $(\hat x, \hat y)$ to be $c_{i-1}$. We query the separation oracle of $C$ at $(\hat x, \hat y)$. If this point is in $C$, we are done. Else, we obtain a separating halfspace $H$. Applying Theorem~\ref{thm:ellipsoid} with $k=n+d$ and $\beta = 0$, we can construct $c_i$ and $E_i$ such that $(c_{i-1} +E_{i-1})\cap H \subseteq c_i + E_i$ and $\vol(E_i) \leq e^{-\frac{1}{5(n+d)}}\vol(E_{i-1}).$ Note that this ensures $C\subseteq c_{i} + E_{i}$ since inductively we know $C\subseteq (c_{i-1} + E_{i-1})\cap H$.
\medskip

\noindent \underline{\bf Case 2: $n \geq 1$.} We compute the projections $c', E'$ of $c_{i-1}, E_{i-1}$ onto the coordinates corresponding to the integers, i.e., $\R^n$. This is easy to do for $c_{i-1}$ (simply drop the other coordinates) and given the matrix $A_{i-1}$ defining $E_{i-1}$, the submatrix $A'$ of $A_{i-1}$ whose rows and columns correspond to the integer coordinates is such that $E_{A'}$ is the projection of $E_{i-1}$. We now solve the closest vector problem (CVP) for $c' \in \R^n$ and the norm given by $A' \in \R^{n\times n}$ using the algorithm in Theorem~\ref{thm:SVP-CVP} to obtain $\hat x \in \Z^n$. 
\medskip

\underline{\bf Case 2a: $\|\hat x - c'\|_{A'} \leq \frac{1}{n+d+1}$.} In other words, $\hat x \in c' + \frac{1}{n+d+1}E'$. Since $c', E'$ are projections of $c_{i-1}, E_{i-1}$ respectively, $c' + \frac{1}{n+d+1}E'$ is the projection of $c_{i-1} + \frac{1}{n+d+1}E_{i-1}$ by linearity of the projection map. Hence, there must be $\hat y \in \R^d$ such that $(\hat x, \hat y) \in c_{i-1} + \frac{1}{n+d+1}E_{i-1}$. Therefore, if we compute $\argmin_{y\in \R^d}\|(\hat x, y) - c_{i-1} \|_{A_{i-1}}$ which amounts to computing the minimizer of an explicit convex quadratic function in $\R^d$ (which can be done analytically or via methods like conjugate gradient), we can find a point $(\hat x, \hat y)$ in $c_{i-1} + \frac{1}{n+d+1}E_{i-1}$. 

We query the separation oracle of $C$ at $(\hat x, \hat y)$. If this point is in $C$, we are done. Else, we obtain a separating halfspace $H$. Since $(\hat x, \hat y)$ is in $c_{i-1} + \frac{1}{n+d+1}E_{i-1}$, this means $c_{i-1} + \frac{1}{n+d+1}E_{i-1}$ is not contained in $H$. Applying Theorem~\ref{thm:ellipsoid} with $k=n+d$ and $\beta = \frac{1}{n+d+1}$, we can construct $c_i$ and $E_i$ such that $(c_{i-1} +E_{i-1})\cap H \subseteq c_i + E_i$ and $\vol(E_i) \leq e^{-\frac{1}{5(n+d)(n+d+1)^2}}\vol(E_{i-1}).$ Note that this ensures $C\subseteq c_{i} + E_{i}$ since inductively we know $C\subseteq (c_{i-1} + E_{i-1})\cap H$.
\medskip

\underline{\bf Case 2b: $\|\hat x - c'\|_{A'} > \frac{1}{n+d+1}$.} In other words, $\hat x \not\in c' + \frac{1}{n+d+1}E'$ which implies that $c' + \frac{1}{n+d+1}E'$ has no integer points since $\hat x$ is the closest integer point to $c'$ in the norm $\| \cdot \|_{A'}$. Theorem~\ref{thm:flatness} implies that there exists $w \in \Z^n\setminus\{0\}$ such that $\max_{x \in E'} \;\langle w, x\rangle - \min_{x \in E'}\;\langle w, x\rangle \leq n(n+d+1).$ Rearranging, this says that $\max_{x,x' \in E'} \langle w, x - x'\rangle \leq n(n+d+1)$ and therefore $$\max_{p \in 2E'} \langle w, p\rangle = \max_{p \in E' + E'} \langle w, p\rangle = \max_{p \in E' - E'} \langle w, p\rangle \leq n(n+d+1),$$ where the equalities follow from the fact that $E'$ is convex and centrally symmetric about the origin. Standard results in convex analysis involving polarity imply that $\| w \|_{\tilde A} = \max_{p \in 2E'} \langle w, p\rangle$ where $\tilde A := \frac{1}{4}A'^{-1}$. We therefore compute the shortest vector $w^\star \in \Z^n\setminus \{0\}$ by the algorithm in Theorem~\ref{thm:SVP-CVP} with respect to the norm $\|\cdot\|_{\tilde A}$ and we are guaranteed that $$\max_{x \in E'} \;\langle w^\star, x\rangle - \min_{x \in E'}\;\langle w^\star, x\rangle \leq n(n+d+1).$$

All mixed-integer points must lie on the hyperplanes $\{(x,y) \in \R^n \times \R^d: \langle w^\star, x\rangle \in \Z\}$. Moreover, since $C \subseteq E$ and $E'$ is the projection of $E$, it suffices to search over the ``slices" of $C$ given by $C \cap \{(x,y) \in \R^n \times \R^d: \langle w^\star, x\rangle = m\}$ for $m = \lceil \langle w^\star, c'\rangle - n(n+d+1)\rceil, \lceil \langle w^\star, c'\rangle - n(n+d+1)\rceil + 1, \ldots, \lfloor \langle w^\star, c'\rangle + n(n+d+1) \rfloor$. By a change of coordinates in the integer constrained variables, these slices involve $n-1$ integer variables and we recurse on these subproblems. We also note that if there exists $z \in C \cap (\Z^n\times \R^d)$ with a ball of radius $\delta$ around $z$ contained in $C$, then the slice containing $z$ will also have the same property. Thus, if the algorithm fails on all the slices, then the algorithm will indeed report correctly that there is no such point in $C$.\footnote{A subtlety here is to make sure that one has access to a separation oracle for the lower dimensional subproblems. This is not hard to implement given access to a separation oracle for $C$: given a point in the new space, one maps back to $\R^n\times \R^d$ and queries the separation oracle there.}

\paragraph{Number of oracle calls and overall complexity.} Within any particular level of the recursion, the algorithm makes at most $5(n+d)(n+d+1)^2\ln\left(\left(\frac{R}{\delta}\right)^{n+d}\right)$ iterations of constructing new ellipsoids in Case 1 or Case 2a. This is because we start with a ball of radius $R$, stop after the volume of the ellipsoid gets smaller than the volume of a ball of radius $\delta$, and the volume is reduced by a factor of at least $e^{-\frac{1}{5(n+d)(n+d+1)^2}}$ every time. Thus, at most $5(n+d)(n+d+1)^2\ln\left(\left(\frac{R}{\delta}\right)^{n+d}\right)$ oracle calls are made within every level of the recursion. The recursion over $2n(n+d+1)$ subproblems leads to at most $(2n(n+d+1))^n = 2^{O(n\log(n+d))}$ subproblems. Putting everything together we obtain the bound stated in the theorem on the number of separation oracle calls. 

There are two computation intensive steps beyond the separation oracle calls: 1) Computing $c_i,E_i$ from $c_{i-1}, E_{i-1}$ based on Theorem~\ref{thm:ellipsoid}, and 2) Solving closest vector and shortest vector problems using the algorithm in Theorem~\ref{thm:SVP-CVP}. The first has complexity $O(n+d)^2$ and the second has complexity $2^{O(n)}$ times a polynomial factor of the sizes of the matrices involved. Both 1) and 2) above may result in irrational numbers if we perform exact computations. Since we wish to remain in the Turing machine model of computation, one has to make sure that we can round these to a polynomial number of bits and the sizes of the numbers do not grow exponentially. We omit these technical details from this presentation and refer the reader to~\cite{GroetschelLovaszSchrijver-Book88}. Once all of this is taken into account, we obtain the bound on the overall complexity of the algorithm stated in the theorem. Using these careful approximation techniques, the space complexity of the algorithm can in fact be bounded by a polynomial in the parameters of the problem (note that the computational (time) complexity is {\em not} polynomial)~\cite{GroetschelLovaszSchrijver-Book88,frank1987application}.\end{proof}

\begin{theorem}\label{thm:opt-alg} Consider the family of problems of the form~\eqref{eq:MICO} such that if $(x^\star, y^\star)\in \Z^n \times \R^d$ is the optimum solution, then there exists $\hat y \in \R^d$ and $\rho > 0$ such that $\{(x,y): \|(x,y) - (x^\star,\hat y)\|_2 \leq \rho\} \subseteq C$, i.e., there is a ``strictly feasible" point $(x^\star, \hat y)$ in the same fiber as the optimum $(x^\star, y^\star)$ with a Euclidean ball of radius $\rho$ in $\R^n \times \R^d$ around $(x^\star, \hat y)$ contained in $C$. Let $\I_{n, d, R,M, \rho}$ be defined as in Section~\ref{sec:information-comp} for this family, except that $\|\cdot \|_2$ is used instead of $\|\cdot\|_\infty$ in the definition of the parameters $R,M, \rho$, and $M$ is defined with respect to the full space $\R^n\times \R^d$, as opposed to just $\R^d$. 

Let the oracle access to an instance $f, C$ in $\I_{n, d, R,M, \rho}$ be through a separation oracle for $C$ and a first-order oracle for $f$. Then for every $\epsilon >0$, there exists an $\epsilon$-approximation algorithm $\cA$ for this problem class with $$\icomp\textstyle{_{\cA}}(n,d,R,M,\rho) \leq 2^{O(n\log(n + d))}\left(\log\left(\frac{MR}{\rho\epsilon}\right)\right)^2$$ and $$\comp\textstyle{_{\cA}}(n,d,R,M,\rho) \leq 2^{O(n\log(n+d))}\poly\left(\log\left(\frac{MR(n+d)}{\rho\epsilon}\right)\right)$$
\end{theorem}

\begin{proof} If $\epsilon> 2MR$, then any feasible solution is an $\epsilon$-approximate solution, so we may simply run the feasibility algorithm from Theorem~\ref{thm:feas-alg} with $\delta:= \rho$. Thus, we assume that $\frac{\epsilon}{2MR} \leq 1$.

We use a standard binary search technique to reduce the problem to a feasibility problem. In particular, we use the algorithm in Theorem~\ref{thm:feas-alg} to test if $C \cap \{z: f(z) \leq \gamma\} = \emptyset$ for some guess $\gamma$ of the optimum value $OPT$. Lemma~\ref{lem:eps-sol-vol} implies that for $\gamma \geq OPT + \frac{\epsilon}{2}$, the set $C \cap \{z: f(z) \leq \gamma\}$ contains a Euclidean ball of radius $\delta:= \frac{\rho\epsilon}{4MR}$ centered at a mixed-integer point in $\Z^n\times \R^d$ (note that because of our strict feasbiility assumption both $z^\star$ and $a$ can be taken as mixed-integer points in the same fiber when applying Lemma~\ref{lem:eps-sol-vol}). Thus, for $\gamma \in [OPT + \frac{\epsilon}{2}, OPT + \epsilon]$, the algorithm in Theorem~\ref{thm:feas-alg} will compute $\hat z \in C \cap (\Z^n \times \R^d)$ with $f(\hat z) \leq \gamma \leq OPT+\epsilon$.

Since the difference between the maximum and the minimum values of $f$ over the Euclidean ball of radius $R$ is at most $2MR$, we need to make at most $\log\left(\frac{4MR}{\epsilon}\right)$ guesses for $\gamma$ in the binary search. The result now follows from the complexity bounds in Theorem~\ref{thm:feas-alg}.\end{proof}

\begin{remark} For ease of exposition, we first presented a feasibility algorithm in Theorem~\ref{thm:feas-alg} and then reduced the optimization problem to the feasibility problem using binary search in Theorem~\ref{thm:opt-alg}. One can do away with the binary search in the following way. If $\epsilon > 2MR$, then any feasible solution is an $\epsilon$-approximate solution, so we may simply run the feasibility algorithm from Theorem~\ref{thm:feas-alg} with $\delta:= \rho$. Otherwise, we follow the feasibility algorithm from Theorem~\ref{thm:feas-alg} with $\delta:= \frac{\rho\epsilon}{2MR}$. Since $\delta \leq \rho$, the algorithm is guaranteed to visit feasible points in Case 1 or 2a of the proof of Theorem~\ref{thm:feas-alg}. Once we find a feasible point, we can query the first-order oracle of the objective function $f$ at this feasible point. Any subgradient inequality/halfspace that shaves off this point satisfies the condition in Theorem~\ref{thm:ellipsoid}, similar to the analysis of Case 1 or 2a in Theorem~\ref{thm:feas-alg}. One appeals to Theorem~\ref{thm:ellipsoid} to obtain a new ellipsoid with reduced volume and the algorithm continues with this new ellipsoid. At the end, the algorithm selects the feasible point with the smallest objective value amongst all the feasible points it visits. This is similar to the idea in the proof of the upper bound in Theorem~\ref{thm:information-comp}. Since the binary search is eliminated, one obtains a slightly better information complexity of $\icomp\textstyle{_{\cA}}(n,d,R,M,\rho) \leq 2^{O(n\log(n + d))}\log\left(\frac{MR}{\rho\epsilon}\right)$.\end{remark}

\begin{remark}\label{rem:pure-cont}[Pure continuous case] Stronger results can be obtained in the pure continuous case, i.e., $n=0$. First, in Case 1 of the algorithm, we use $\beta = 0$ instead of $\beta = \frac{1}{n+d+1}$, reducing the volume of the ellipsoid by a factor $e^{-\frac{1}{5(n+d)}}$ every time. Thus we make a factor of $(n+d+1)^2$ less number of iterations in Case 1 of the proof of Theorem~\ref{thm:feas-alg}. Moreover, there is no recursion needed and thus, the algorithm's information complexity is $O\left(d^2\log\left(\frac{MR}{\rho\epsilon}\right)\right)$ with an additional computational overhead of $O(d^2)$ for computing the new ellipsoids. This is the classical {\em ellipsoid algorithm} for convex optimization. Thus, one obtains an $\epsilon$-approximation algorithm for the optimization problem that differs only by a factor of the dimension $d$ from the $\epsilon$-information complexity bound given in Theorem~\ref{thm:information-comp}\footnote{There is a slight discrepancy because of the use of the $\|\cdot \|_\infty$-norm for the information complexity bound (see Theorem~\ref{thm:information-comp}), and the use of $\|\cdot \|_2$-norm here. This adds a $\log(d)$ factor to the complexity of the ellipsoid algorithm, compared to the information complexity bound. We are not aware of any work that resolves this discrepancy.}. Vaidya~\cite{Vaidya1996} designed an algorithm whose information complexity matches Theorem~\ref{thm:information-comp}'s $\epsilon$-information complexity bound of $O\left(d\log\left(\frac{MR}{\rho\epsilon}\right)\right)$, with the same overall complexity as the ellipsoid algorithm. See~\cite{anstreicher1997vaidya,jiang2020improved,lee2015faster} for improvements on the overall complexity of Vaidya's algorithm. Lemma~\ref{thm:ellipsoid} with $\beta > 0$ is also used in continuous convex optimization under the name of the {\em shallow cut ellipsoid method}; see~\cite{GroetschelLovaszSchrijver-Book88} for details.
\end{remark}

\begin{remark}\label{rem:pure-int}[Pure integer case] For the pure integer case, i.e., $d = 0$ one can strengthen both Theorems~\ref{thm:feas-alg} and~\ref{thm:opt-alg} by removing the ``strict feasibility" type assumptions. In particular, one can prove a variant of Theorem~\ref{thm:feas-alg} with an exact feasibility algorithm that either reports a point in $C\cap \Z^n$ or correctly decides that $C \cap \Z^n = \emptyset$. One observes that if the volume of the ellipsoid in Case 2a falls below $\frac{1}{n!}$, one can be sure that all integer points in $C$ lie on a single hyperplane. This is because otherwise there are affinely independent points $x_1, \ldots, x_{n+1} \in C \cap \Z^n$ and the convex hull of these points has volume at least $\frac{1}{n!}$. Thus,  we can recurse on the lower dimensional problem. For more details see~\cite{Dadush12}. Another approach is to simply stop the iterations in Case 2a when the ellipsoid has volume less than 1. Then one can show that there is a translate of this ellipsoid that does not intersect $\Z^n$. Applying Theorem~\ref{thm:flatness}, one can again find $n$ lower dimensional slices to recurse on. This idea was explored in~\cite{OertelWagnerWeismantel14} for polyhedral outer and inner approximations. We thus obtain an exact optimization algorithm with information complexity $2^{O(n\log(n))}\log(R)$ and overall complexity $2^{O(n\log(n))}\poly(\log(nR))$.
\end{remark}

\begin{remark} The information or overall complexity bounds presented in Theorems~\ref{thm:feas-alg} and~\ref{thm:opt-alg} are not the best possible ones. There is a general consensus in the discrete optimization community that the right bound is $2^{O(n\log n)}\poly\left(d,\log\left(\frac{MR}{\rho\epsilon}\right)\right)$. Thus, the dependence on the dimensions $n$ (number of integer variables) and $d$ (number of continuous variables) is $2^{O(n\log n)}\poly(d)$ instead of $2^{n\log(n+d)} = (n+d)^{O(n)}$. In other words, the degree of the polynomial function of $d$ is independent of $n$ in the new stated bound. 

How can this be achieved? Observe that if one could work with simply the projection of the convex set on to the space of the integer variables, then one can reduce the problem to the pure integer case discussed in Remark~\ref{rem:pure-int} (assuming one has at least some integer constrained variables; otherwise, one defaults to Remark~\ref{rem:pure-cont} for the continuous case). Indeed, this was the idea originally presented for the mixed-integer {\em linear} case in Lenstra's paper~\cite{Lenstra83}. In the general nonlinear setting, this can be achieved if one can design a separation oracle for projections of convex sets, given access to separation oracles for the original set, that runs in time polynomial in $n,d$. This can be done via a result that is colloquially called ``equivalence of separation and optimization". This circle of ideas roughly says the following: given access to a separation oracle to a convex set, one can optimize linear functions over it in time that is polynomial in the dimension of the convex set (and the parameters $R, \rho, \epsilon$ and the objective vector size), and conversely, if one can optimize over the set one can implement a separation oracle by making polynomially many calls to the optimization oracle. The first part of this equivalence is simply a restatement of Theorems~\ref{thm:feas-alg} and~\ref{thm:opt-alg}; in fact, one is only concerned with the continuous case. We refer the reader to~\cite{lovasz1986algorithmic,GroetschelLovaszSchrijver-Book88} for details on the full equivalence. Coming back to projections: using the separation oracle for the original set, one can implement an optimization oracle for it. This optimization oracle gives an optimization oracle over the projection since optimizing a linear function over the projection is the same as optimizing over the original set. Using the equivalence, this gives a separation oracle for the projection. Now one can appeal to the arguments in Remark~\ref{rem:pure-int}.

However, all of these arguments are quite delicate and necessarily require very careful approximations. Thus, while these arguments should in principle work, the author is not aware of any source in the literature where all the tedious details have been fully worked out, except in the rational, linear case from Lenstra's original paper~\cite{Lenstra83}. Our exposition here is considerably simpler because it avoids these technicalities, but this is at the expense of the weaker bounds stated in Theorems~\ref{thm:feas-alg} and~\ref{thm:opt-alg}. In Lenstra's original way of doing this, the equivalence of separation and optimization is not needed since he works directly with an optimization oracle for (rational) linear programming for which an exact polynomial time algorithm has been known since Khachiyan's work~\cite{Khachiyan}, which builds on the ellipsoid algorithm discussed in Remark~\ref{rem:pure-cont}. See~\cite{Lenstra83} for more details.
\end{remark}

\paragraph{Brief historical comments.} The ideas presented in this section are refinements and improvements over seminal ideas of Lenstra~\cite{Lenstra83}, and hence our tribute in the title of this section. His original 1983 paper investigated the mixed-integer {\em linear} case, i.e., when $C$ is a polytope and $f$ is a linear function~\cite{Lenstra83}. His insights were soon extended to handle the general nonlinear case in~\cite{GroetschelLovaszSchrijver-Book88}. Kannan~\cite{Kannan87} achieved a breakthrough in the complexity bounds -- improving from $2^{O(n^3)}$ dependence on the number of integer variables to $2^{O(n\log n)}$ -- by modifying the algorithm to recurse over lower dimensional affine spaces, as opposed to hyperplanes as discussed above. We refer to~\cite{Dadush12,Heinz05,KhachiyanPorkolab00,hildebrand2013new} for a representative sample of important papers since then. See also Fritz Eisenbrand's excellent survey chapter in~\cite{Eisenbrand2010}. To the best of the author's knowledge, in the pure integer case the sharpest constant in the exponent of $2^{O(n\log n)}$ is derived in Daniel Dadush's Ph.D. thesis. This requires the use of highly original and technically deep ideas~\cite{Dadush12}.

\subsection{Pruning, nontrivial cutting planes and branch-and-cut}\label{sec:branch-and-cut}

The algorithm presented in Section~\ref{sec:lenstra} utilizes only trivial cutting planes (see Definition~\ref{def:branching-and-cuts}) and solves the optimization problem by reducing to the feasibility problem via binary search. Modern solvers for mixed-integer optimization utilize nontrivial cutting planes and also use a crucial ingredient called {\em pruning}. We now present the general framework of {\em branch-and-cut} methods which incorporate both these techniques. The algorithms in Section~\ref{sec:lenstra} will then be seen to be essentially special instances of such methods.

\begin{definition} A family $\D$ of disjunctions is called a {\em branching scheme}. A {\em cutting plane paradigm} is a map $\CP$ that takes as input any closed, convex set $C$ and $\CP(C)$ is a family of cutting planes for $C \cap (\Z^n\times \R^d)$. $\CP(C)$ may contain trivial and/or nontrivial cutting planes, and may even be empty for certain inputs $C$.\end{definition}

\begin{example}\label{ex:CG-disj}
\begin{enumerate}
\item {\em Chv\'atal-Gomory cutting plane paradigm:} Given any convex set $C\subseteq \R^n \times \R^d$, define $$\CP(C):=\{H' : H' = \conv(H \cap (\Z^n \times \R^d)), \; H \textrm{ rational halfspace with }H \supseteq C\}.$$

$H' \in \CP(C)$ is nontrivial if and only if $H$ is of the form $\{(x,y)\in \Z^n \times \R^d: \langle a, x \rangle \leq b\}$ for some $a \in \Z^n$ with relatively prime coordinates and $b \not\in \Z$, in which case $H' = \{(x,y)\in \Z^n \times \R^d: \langle a, x \rangle \leq \lfloor b\rfloor\}$. 
\item {\em Disjunctive cuts:} Given any family of disjunctions (branching scheme) $\D$, the {\em disjunctive cutting plane paradigm based on $\D$} is defined as $$\CP(C):=\{H' \textrm{ halfspace}: H' \supseteq C \cap D, \; D \in \D\}.$$ The collection of halfspaces $H'$ valid for $C \cap D$ are said to be the {\em cutting planes derived from the disjunction $D$}. These are valid cutting planes since $\Z^n \times \R^d \subseteq D$ by definition of a disjunction, and therefore $C\cap (\Z^n\times \R^d) \subseteq C \cap D \subseteq  H'$. A disjunction $D$ produces nontrivial cutting planes for a compact, convex set $C$ if and only if at least one extreme point of $C$ is not contained in $D$.

 \end{enumerate}
\end{example}

\begin{table}[htbp]
\begin{tabular}{l}
\hline
{\bf Branch-and-cut framework based on a branching scheme $\D$ and cutting plane paradigm $\CP$} \\
\hline
\\
{\bf Input:} A closed, convex set $C \subseteq \R^n \times \R^d$, a convex function $f: \R^n\times \R^d \to \R$, error guarantee $\epsilon > 0$, and \\
a relaxation $X\subseteq \R^n \times \R^d$ which is closed, convex and contains $C$.
\\
{\bf Output:} A point $z^\star \in C \cap (\Z^n \times \R^d)$ such that $f(z^\star) \leq OPT + \epsilon$, where $OPT = \inf\{f(z): z \in C\}$.
\\
\\
1. Initialize a set $L = \{X\}$. 
Initialize $UB = +\infty$.\\
2. While $L \neq \emptyset$ do: \\
\hspace{1cm} a. [Node selection] Select an element $N \in L$ and update $L:= L\setminus\{N\}$. \\
\hspace{1cm} b. [Pruning] If it can be verified that $\inf\{f(z) : z \in C\cap N\} \geq UB - \epsilon$, then continue the While loop.\\
\hspace{1.45cm} Else, select a test point $\hat z$ in $N$.\\
\hspace{1cm} c. If $\hat z \in C \cap(\Z^n \times \R^d)$, obtain a subgradient $h \in \partial f(\hat z)$ and add the subgradient halfspace \\
\hspace{1.45cm} $H = \{z: \langle h, z - \hat z \rangle \leq 0\}$ to all the elements in $L$, i.e., update $N:= N \cap H$ for all $N \in L$.\\
\hspace{1.45cm} Additionally, if $f(\hat z) < UB$, then update $UB = f(\hat z)$ and $z^\star = \hat z$.\\
\hspace{1cm} d. If $\hat z \not\in C \cap(\Z^n \times \R^d)$, decide whether to BRANCH or CUT. \\
\\
\hspace{2cm} If BRANCH, then choose a disjunction $D = Q_1 \cup \ldots \cup Q_k$ in $\D$ such that $\hat z \not\in D$. \\
\hspace{2cm} Select sets (relaxations) $N_1, \ldots, N_2$ such that $N\cap Q_i \subseteq N_i$. \\
\hspace{2cm} Update $L := L \cup \{N_1, \ldots, N_k\}$. \\
\\
\hspace{2cm} If CUT, then choose a cutting plane $H \in \CP(C \cap N)$ such that $\hat z \not\in H$. \\
\hspace{2cm} Select a set $N'$ such that $N\cap H \subseteq N'$. Update $L := L \cup \{N'\}$. \\
\\
\hline
\end{tabular}
\end{table}

\begin{remark} We obtain a specific branch-and-cut procedure once we specify the following things in the framework above.
\begin{enumerate}
\item In Step 2a., we must decide on a strategy to select an element from $L$. In the case of the algorithms presented in Section~\ref{sec:lenstra}, this would be the choice of a ``slice" to recurse on.
\item In Step 2b., we must decide on a strategy to verify the condition $\inf\{f(z) : z \in N\} \geq UB + \epsilon$. In the case of the algorithms presented in Section~\ref{sec:lenstra}, this is determined by a volume condition on $N$ (which is an ellipsoid). Another common strategy is used in linear integer optimization, where $C \cap N$ is a polyhedron and linear optimization methods like the simplex method or an interior-point algorithm is used to determine $\inf\{f(z) : z \in C\cap N\}$. More generally, one could have a convex optimization subroutine suitable for the class of problems under study.
\item In Step 2b., $\inf\{f(z) : z \in C \cap N\} < UB + \epsilon$ and one must select a test point $\hat z \in N$ and one must have a procedure/subroutine for this. In the algorithms presented in Section~\ref{sec:lenstra}, this was chosen as the center of the ellipsoid in Step I, and in Step II it was chosen using the CVP subroutine (and a convex quadratic minimization over the corresponding fiber if the CVP returned a point in the inner ellipsoid). In most solvers, this test point is taken as an optimal or $\epsilon$-approximate solution to the convex optimization problem $\inf\{f(z) : z \in C\cap N\}$ or $\inf\{f(z) : z \in N\}$.
\item In Step 2d., one must have a strategy for deciding whether to branch or to cut, and in either case have a strategy for selecting a disjunction or a cutting plane. The decision to branch might fail because there is no disjunction $D$ in the chosen branching scheme $\D$ that does not contain $\hat z$. In such a case, we simply continue the While loop. In the algorithms from Section~\ref{sec:lenstra}, the disjunction family used was the split disjunctions defined in Example~\ref{ex:Disj-CP-ex}: the ``slices" can be seen as branching on the disjunctions $\{(x,y): \langle w, x \rangle \leq j\} \cup \{(x,y): \langle w, x \rangle \geq j+1\}$.

If the decision is to add a cutting plane, one may add a trivial cutting plane valid for $C\cap N$, as was done in the algorithms in Section~\ref{sec:lenstra}. One may also fail to find a cutting plane that removes $\hat z$, because either the cutting plane paradigm can produce no such cutting plane, i.e., $\CP(C\cap N) = \emptyset$, or because the strategy chosen fails to find such a cutting plane in $\CP(C\cap N)$ even though one exists. In such a case, we simply continue the While loop.

Finally, in Step 2d., we must have a strategy to select the relaxations $N_1, \ldots, N_k$ if the decision is to branch, or we must have a strategy to select a relaxation $N'$ if the decision is to cut. In the algorithms in Section~\ref{sec:lenstra}, these relaxations were taken as ellipsoids.
\item If an algorithm based on the branch-and-cut framework above decides never to branch in Step 2d., it is called a {\em pure cutting plane} algorithm. If an algorithm decides never to use cutting planes in Step 2d., it is called a {\em pure branch-and-bound} algorithm.
\end{enumerate}
\end{remark}

\begin{remark} In most solvers, the relaxations $N_i$ are simply taken to be $N \cap Q_i$ in a decision to branch, and the relaxation $N'$ is simply taken to be $N\cap H$ in a decision to cut. However, as mentioned above, in the algorithms from Section~\ref{sec:lenstra} we consider ellipsoidal relaxations of the $N\cap H$ after a (trivial) cutting plane is added, and ellipsoidal relaxations of the ``slices".
\end{remark}
\bigskip

\paragraph{Does this help?} In practice, pruning and nontrivial cutting planes make a huge difference~\cite{bixby2012brief,Bixby,lodi2010mixed}. Turning these off will bring most of the solvers to a grinding halt on even small scale problems. Nevertheless, from a theoretical perspective, researchers have not been able to improve on the $2^{O(n\log(n+d))}$ algorithm from Section~\ref{sec:lenstra} by utilizing pruning and nontrivial cutting planes for the general problem; see~\cite{basu-split-2D,dey2020branch,naderi2021worst} for examples of positive results in restricted settings. Another empirical fact is that if branching is completely turned off and only cutting planes are used, then again the solvers' performance degrades massively. Recently, some results have been obtained that provide some theoretical basis to these empirical observations that the combination of branching and cutting planes performs significantly better than branching alone or using cutting planes alone. We present some of these results now.

The next definition is inspired by the following simple intuition. It has been established that certain branching schemes can be simulated by certain cutting plane paradigms in the sense that for the problem class under consideration, if we have a pure branch-and-bound algorithm based on the branching scheme, then there exists a pure cutting plane algorithm for the same class that has complexity at most a polynomial factor worse than the branch-and-bound algorithm. Similarly, there are results that establish the reverse. See~\cite{dash2002exponential,dash2005exponential,basu-BB-CP,basu-BB-CP-II,fleming2021power,beame_et_al:LIPIcs:2018:8341}, for example. In such situations, combining branching and cutting planes into branch-and-cut is likely to give no substantial improvement since one method can always do the job of the other, up to polynomial factors.

\begin{definition}\label{complementary-pair} Let $\I$ be a family of mixed-integer convex optimization problems of the form~\eqref{eq:MICO}, along with a size hierarchy (see Definition~\ref{def:size}). To make the following discussion easier, we assume that the objective function is linear. This is without loss of generality since we can introduce an auxiliary variable $v$, introduce the epigraph constraint $f(z) \leq v$ and use the linear objective ``$\inf v$".

A cutting plane paradigm $\CP$ and a branching scheme $\D$ are {\em complementary for $\I$} if there is a family of instances $\I_{\CP>\D}\subseteq \I$ such that there is a pure cutting plane algorithm based on $\CP$ that has polynomial (in the size of the instances) complexity and any branch-and-bound algorithm based on $\D$ is exponential (in the size of the instances), and there is another family of instances $\I_{\CP<\D}\subseteq \I$ where $\D$ gives a polynomial complexity pure branch-and-bound algorithm while any pure cutting plane algorithm based on $\CP$ is exponential. \end{definition}
\medskip

We wish to formalize the intuition that branch-and-cut is expected to be exponentially better than branch-and-bound or cutting planes alone for complementary pairs of branching schemes and cutting plane paradigms. But we need to make some mild assumptions about the branching schemes and cutting plane paradigms. {\em All known branching schemes and cutting plane methods from the literature satisfy the following conditions}.

\begin{definition}\label{def:regular} A branching scheme is said to be {\em regular} if no disjunction involves a continuous variable, i.e., each polyhedron in the disjunction is described using inequalities that involve only the integer constrained variables.

A branching scheme $\D$ is said to be {\em embedding closed} if disjunctions from higher dimensions can be applied to lower dimensions. More formally, let $n_1$, $n_2$, $d_1$, $d_2 \in \N$. If $D \in \D$ is a disjunction in $\R^{n_1} \times \R^{d_1} \times  \R^{n_2} \times \R^{d_2}$ with respect to $\Z^{n_1} \times \R^{d_1} \times \Z^{n_2} \times \R^{d_2}$, then the disjunction $D\cap (\R^{n_1} \times \R^{d_1} \times \{0\}^{n_2} \times \{0\}^{d_2})$, interpreted as a set in $\R^{n_1} \times \R^{d_1}$, is also in $\D$ for the space $\R^{n_1}\times \R^{d_1}$ with respect to $\Z^{n_1} \times \R^{d_1}$ (note that $D\cap (\R^{n_1} \times \R^{d_1} \times \{0\}^{n_2} \times \{0\}^{d_2})$, interpreted as a set in $\R^{n_1} \times \R^{d_1}$, is certainly a disjunction with respect to $\Z^{n_1} \times \R^{d_1}$; we want $\D$ to be closed with respect to such restrictions).

A cutting plane paradigm $\CP$ is said to be {\em regular} if it has the following property, which says that adding ``dummy variables" to the formulation of the instance should not change the power of the paradigm. Formally, let $C \subseteq \R^{n} \times \R^{d}$ be any closed, convex set and let $C' = \{(x,t) \in \R^{n}\times \R^{d} \times \R: x \in C,\;\; t = \langle f, x \rangle\}$ for some $f \in \R^n$. Then if a cutting plane $\langle a, x \rangle \leq b$ is derived by $\CP$ applied to $C$, i.e., this inequality is in $\CP(C)$, then it should also be in $\CP(C')$, and conversely, if $\langle a, x\rangle + \mu t \leq b$ is in $\CP(C')$, then the equivalent inequality $\langle a + \mu f, x \rangle \leq b$ should be in $\CP(C)$.

A cutting plane paradigm $\CP$ is said to be {\em embedding closed} if cutting planes from higher dimensions can be applied to lower dimensions. More formally, let $n_1, n_2, d_1, d_2 \in \N$. Let $C \subseteq \R^{n_1} \times \R^{d_1}$ be any closed, convex set. If the inequality $\langle c_1, x_1\rangle + \langle a_1, y_1\rangle + \langle c_2, x_2\rangle + \langle a_2, y_2\rangle \leq \gamma$ is a cutting plane for $C \times \{0\}^{n_2} \times \{0\}^{d_2}$ with respect to $\Z^{n_1} \times \R^{d_1} \times \Z^{n_2} \times \R^{d_2}$ that can be derived by applying $\CP$ to $C \times \{0\}^{n_2} \times \{0\}^{d_2}$, then the cutting plane $\langle c_1, x_1\rangle+ \langle a_1, y_1\rangle \leq \gamma$ that is valid for $C \cap (\Z^{n_1} \times \R^{d_1})$ should also belong to $\CP(C)$.

A cutting plane paradigm $\CP$ is said to be {\em inclusion closed}, if for any two closed convex sets $C \subseteq C'$, we have $\CP(C') \subseteq \CP(C)$. In other words, any cutting plane derived for $C'$ can also be derived for a subset $C$.
\end{definition}

\begin{theorem}\label{thm:BC>BB/CP}~\cite[Theorem 1.12]{basu-BB-CP-II} Let $\D$ be a regular, embedding closed branching scheme and let $\CP$ be a regular, embedding closed, and inclusion closed cutting plane paradigm such that $\D$ includes all variable disjunctions and $\CP$ and $\D$ form a complementary pair for a mixed-integer convex optimization problem class $\I$. Then there exists a family of instances in $\I$ such that there exists a polynomial complexity branch-and-cut algorithm, whereas any branch-and-bound algorithm based on $\D$ and any cutting plane algorithm based on $\CP$ are of exponential complexity.
\end{theorem}

The rough idea of the proof of Theorem~\ref{thm:BC>BB/CP} is to embed pairs of instances from $\I_{\CP>\D}$ and $\I_{\CP<\D}$ as faces of a convex set such that a single variable disjunction results in these instances as the subproblems in a branch-and-cut algorithm. On one subproblem, one uses cutting planes and on the other subproblem one uses branching. However, since $\CP$ and $\D$ are complementary, a pure cutting plane or pure branch-and-bound algorithm takes exponential time in processing one or the other of the faces. The details get technical and the reader is referred to~\cite{basu-BB-CP-II}.

\begin{example}\label{ex:complementary} We now present a concrete example of a complementary pair that satisfies the other conditions of Theorem~\ref{thm:BC>BB/CP}. Let $\I$ be the family of mixed-integer linear optimization problems described in point 2. of Example~\ref{ex:concrete-ex}, with standard ``binary encoding" oracles described in point 2. of Example~\ref{ex:concrete-oracles} and size hierarchy as defined in point 2. of Example~\ref{ex:concrete-size}. Let $\CP$ to be the Chv\'atal-Gomory paradigm (point 1. in Example~\ref{ex:CG-disj}) and $\D$ to be the family of variable disjunctions (point 1. in Example~\ref{ex:Disj-CP-ex}). They are both regular and $\D$ is embedding closed. The Chv\'atal-Gomory paradigm is also embedding and inclusion closed. 

Consider the so-called ``Jeroslow instances": For every $n\in \N$, $\max\{\sum_{i=1}^n x_i : \sum_{i=1}^n x_i \leq \frac{n}{2},\;\; x \in [0,1]^n, \;\; x\in \Z^n\}$. The single Chv\'atal-Gomory cut $\sum_{i=1}^n x_i \leq \lfloor \frac{n}{2} \rfloor$ proves optimality, whereas any branch-and-bound algorithm based on variable disjunctions has complexity at least $2^{\lfloor \frac{n}{2}\rfloor}$~\cite{Jeroslow1974}. On the other hand, consider the set $T_h \in \R^2$, where $T_h = \conv\{(0,0), (1,0), (\frac12,h)\}$ and consider the family of problems for $h \in \N$: $\max\{x_2: x \in T, \;\; x\in \Z^2\}$. Any Chv\'atal-Gomory paradigm based algorithm has exponential complexity in the size of the input, i.e., every proof has length at least $\Omega(h)$~\cite{sch}. On the other hand, a single disjunction on the variable $x_1$ solves the problem.

\end{example}

Example~\ref{ex:complementary} shows that the classical Chv\'atal-Gomory cuts and variable branching are complementary and thus Theorem~\ref{thm:BC>BB/CP} implies that they give rise to a superior branch-and-cut routine when combined, compared to their stand-alone use. The Chv\'atal-Gomory cutting plane paradigm and variable disjunctions are the most widely used pairs in state-of-the-art branch-and-cut solvers. We thus have some theoretical basis for explaining the success of this particular combination.

In~\cite{basu-BB-CP,basu-BB-CP-II}, the authors explore whether certain widely used cutting plane paradigms and branching schemes form complementary pairs. We summarize their results here in informal terms and refer the two papers cited for precise statements and proofs.

\begin{enumerate}
\item {\em Lift-and-project} cutting planes (disjunctive cutting planes based on variable disjunctions -- see point 2. in Example~\ref{ex:CG-disj}) and variable disjunctions are {\em not} a complementary pair for $0/1$ pure integer convex optimization problems, i.e., $C \subseteq [0,1]^n$. Any branch-and-bound algorithm based on variable disjunctions can be simulated by a cutting plane algorithm with the same complexity\footnote{There is a technical problem that arises here between the notions of algorithm and {\em proof}. We have omitted all discussions of cutting plane and branch-and-bound proofs here, which are powerful tools to prove unconditional lower bounds on these algorithms. The precise statement is that any branch-and-bound proof based on variable disjunctions can be replaced by a lift-and-project cutting plane proof of the same size. See~\cite{basu-BB-CP} for details.}. Moreover, there are instances of graph stable set problems where there is a lift-and-project cutting plane algorithm with polynomial complexity, but any branch-and-bound algorithm based on variable disjunctions has exponential complexity. See Theorems 2.1 and 2.2 in~\cite{basu-BB-CP}, and results in~\cite{dash2002exponential,dash2005exponential}.
\item Lift-and-project cutting planes and variable disjunctions {\em do form} a complementary pair for {\em general} mixed-integer convex optimization problems, i.e., $C$ is not restricted to be in the $0/1$ hypercube\footnote{``Lift-and-project cuts" here mean disjunctive cutting planes based on the variable disjunctions (typically the phrase ``lift-and-project" is reserved for $0/1$ problems).}. See Theorems 2.2 and 2.9 in~\cite{basu-BB-CP}.
\item Split cutting planes  (disjunctive cutting planes based on split disjunctions -- see point 2. in Example~\ref{ex:CG-disj}) and split disjunctions are {\em not} a complementary pair for general pure integer convex problems. Any cutting plane algorithm based on split cuts can be simulated by a branch-and-bound algorithm based on split disjunctions with the same complexity (up to constant factors)\footnote{The same caveat as in point 1. regarding algorithms versus proofs applies.}. See Theorem 1.8 in~\cite{basu-BB-CP-II}.
\end{enumerate}

\paragraph{Connections to proof complexity.} Obtaining concrete lower bounds for branch-and-cut algorithms has a long history within the optimization, discrete mathematics and computer science communities. In particular, there is a rich interplay of ideas between optimization and proof complexity arising from the fact that branch-and-cut can be used to certify emptiness of sets of the form $P \cap \{0,1\}^n$, where $P$ is a polyhedron. This question is of fundamental importance in computer science because the {\em satisfiability} question in logic can be modeled in this way. We provide here a necessarily incomplete but representative sample of references for the interested reader~\cite{beame_et_al:LIPIcs:2018:8341,dadush2020complexity,dey2021lower,dash2002exponential,dash2005exponential,dash2010complexity,chvatal1989cutting,chvatal1984cutting,chvatal1980hard,cook2001matrix,bockmayr1999chvatal,eisenbrand2003bounds,rothvoss20130,bonet1997lower,razborov2017width,impagliazzo1994upper,buss1996cutting,cook1987complexity,goerdt1990cutting,goerdt1991cutting,clote1992cutting,pudlak1997lower,pudlak1999complexity,krajivcek1998discretely,grigoriev2002complexity,fleming2021power,cook1990complexity}.

\section{Discussion and open questions}

Our presentation above necessarily selected a small subset of results in the vast literature on the complexity of optimization algorithms, even restricted to convex mixed-integer optimization. We briefly discuss three major areas that were left out of the discussion above.

\paragraph{Mixed-integer linear optimization.} If we consider the problem class discussed in point 2. of Example~\ref{ex:concrete-ex}, with algebraic oracles as described in point 2. of Example~\ref{ex:concrete-oracles}, then our $\epsilon$-information complexity bounds from Theorem~\ref{thm:information-comp} to do not apply anymore. Firstly, we have a much more restricted class of problems. Secondly, the algebraic oracles seem to be more powerful than separation oracles in the following precise sense. Using the standard size hierarchy on this class based on ``binary encodings" discussed in point 2. of Example~\ref{ex:concrete-size}, a separation/first-order oracle can be implemented easily with the standard algebraic oracle from Example~\ref{ex:concrete-oracles} in polynomial time, but it is not clear if a separation oracle can implement the algebraic oracle in polynomial time (see point 4. under ``Open Problems" below). Moreover, the $\epsilon$-information complexity with respect to the algebraic oracle is bounded by the size of the instance, because once we know all the entries of $A,B, b,c$, there is no ambiguity in the problem anymore. Nevertheless, it is well-known that the problem class is $NP$-hard~\cite{GareyJohnson-Book79}. Therefore, unless $P = NP$, the computational complexity of any algorithm for this class is not polynomial and under the so-called {\em exponential time hypothesis (ETH)}, it is expected to be exponential. 

Nevertheless, several rigorous complexity bounds have been obtained with respect to difference parameterizations of this problem class. We summarize them here with pointers to the references. We note here that the algorithms presented in Section~\ref{sec:lenstra} are variants of an algorithm that was designed by Lenstra for this problem class of mixed-integer linear optimization~\cite{Lenstra83}. We discuss below approaches that are completely different in nature.

\begin{enumerate}
\item {\em Dynamic programming based algorithms.} We restrict to pure integer instances, i.e., $d=0$ and assume all the entries of $A,b,c$ are integer. We parameterize such instances by three parameters $n,m,\Delta, W$. $\I_{m,n,\Delta,W}$ are those instances with dimension is at most $n$, where $A$ has at most $m$ rows, the absolute values of the entries of $A$ are bounded by $\Delta$, and the absolute values of $b$ are bounded by $W$. In this setting, Papadimitriou designed a dynamic programming based algorithm with complexity $O((n+m)^{2m+2}(m\cdot\max\{\Delta,W\})^{(m+1)(2m+1)})$~\cite{papadimitriou1981complexity}. This was improved recently to $O((m\Delta)^m\cdot (n+m)^3W)$ by Eisenbrand and Weismantel by a clever use of the so-called Steinitz lemma~\cite{steinitz1913bedingt}. Subsequent improvements were achieved by Jansen and Rohwedder~\cite{jansen2018integer}. Matching lower bounds, subject to the exponential time hypothesis, have been established in~\cite{fomin2016fine,knop2020tight}.
\item {\em Fixed subdeterminants.} Restricting to pure integer instances with integer data as in the previous point, another paramterization that has been considered is by $n,m,\Delta$, where $\I_{m,n,\Delta}$ are those instances with dimension is at most $n$, where $A$ has at most $m$ rows, the absolute values of $n\times n$ subdeterminants of $A$ are bounded by $\Delta$. When $\Delta=1$, a classical result in integer optimization shows that one can simply solve the linear optimization problem to obtain the optimal integer solution~\cite{sch}. In 2009, Veselov and Chirkov showed that the feasibility problem can be solved in polynomial time when $\Delta = 2$~\cite{veselov2009integer}, and the optimization version was resolved using deep techniques from combinatorial optimization by Artmann, Weismantel and Zenkulsen in~\cite{artmann2017strongly}. See also related results for general $\Delta$ in~\cite{artmann2016note,basu2021enumerating,paat2021integrality,gribanov2016width,gribanov2016integer,gribanov2020note}.
\item {\em Parameterizations based on the structure of $A$.} Over the past 25 years, a steady line of research has used algebraic and combinatorial techniques to design algorithms for pure integer optimization that exploit the structure of the constraint matrix $A$. Several parameters of interest have been defined based on different structural aspects of $A$. Listing all the literature here is impossible. Instead, we point to~\cite{eisenbrand2019algorithmic,onn2010nonlinear,de2012algebraic,Bertsimas2005,cunningham2007integer,margulies2013cunningham} and the references therein as excellent summaries and starting points for exploring this diverse field of research activity. An especially intriguing aspect of these algorithms is that they search for the optimal solution by iteratively moving from one feasible solution to a better one. In contrast, the ``Lenstra-style" algorithm presented in Section~\ref{sec:lenstra} approaches the optimal solution ``from outside" in a sense by constructing outer (ellipsoidal) approximations. Thus, the newer algebraic and combinatorial algorithms are more in the spirit of {\em interior point methods} in nonlinear optimization. Exploring the possibility of a unified ``interior point style" algorithm for convex mixed-integer optimization would contribute further to bridging the continuous and discrete sides of mathematical optimization.\end{enumerate}

\paragraph{Mixed-integer polynomial optimization.} Instead of linear constraints and objectives as in MILPs, one can consider the problem with polynomial constraints and objective. The standard oracle is algebraic as well in the sense that one can query for the value of a coefficient in a constraint or the objective. The literature on this problem is vast and touches on classical fields like algebraic geometry and diophantine equations. In fact, Hilbert's 10th problem in his famous list of 23 problems presented before the International Congress of Mathematicians in 1900 asks for the construction of an algorithm that solves polynomial equations in integer variables~\cite{hilbert23}. This problem was proven to be undecidable in a line of work spanning several decades~\cite{matiyasevich1993}. Thus, while the completely general mixed-integer polynomial optimization problem cannot be solved algorithmically, an enormous literature exists on restricted versions of the problem. For example, with no integer variables we enter the realm of real algebraic geometry or the existential theory of reals where the problem is decidable; see~\cite{BasuPollackRoyRAG} with references to a rich literature. Thus, if one has or can infer finite bounds on the integer constrained decision variables, then the undecidability issues go away since one can enumerate over the integer variables and reduce to the existential theory of reals. The literature is too vast to summarize here; instead, we point the reader to the excellent survey~\cite{koppe2012complexity} and the textbook expositions in~\cite{de2012algebraic,onn2010nonlinear,Bertsimas2005}, along with the references therein. Some recent progress has appeared in~\cite{del2016minimizing,hildebrand2016fptas,bienstock2021complexity,DelPiaWeismantel,del2018approximation,pia2017mixed,del2017polyhedral,del2018multilinear,del2020impact,del2019subdeterminants,del2021running,del2022complexity,del2022simple}.

\paragraph{Continuous convex optimization.} The information and algorithmic complexity of continuous convex optimization presented in Sections~\ref{sec:information-comp} and~\ref{sec:alg-comp} barely touch the vast literature in this area. For instance, instead of parameterizing by $d,M,R,\rho$, modern day emphasis is placed on ``dimension independent" complexity. For instance, if one focuses on unconstrained minimization instances and parameterizes the class with only two parameters: a parameter $M$ that is usually related to the Lipschitz constant of the objective function or its derivatives (if we further restrict to (twice) differentiable functions), and $R$ is a parameter that says the optimal solution is contained in a Euclidean ball of radius $R$. It is important to note that the dimension is not part of the parameterization. In the nonsmooth case, one can establish matching upper and lower bounds of $O\left(\frac{MR}{\epsilon^2}\right)$ on the information complexity of a broad class of iterative $\epsilon$-approximation algorithms. These can be improved to $O\left(\frac{MR}{\sqrt{\epsilon}}\right)$ in the smooth case. We will fail to do justice to this enormous and active area of research in this manuscript and instead point the reader to the excellent monographs~\cite{nesterov2018lectures,bubeck2014convex}.

\paragraph{Nonconvex continuous optimization.} The complexity landspace for nonconvex continuous optimization has recently seen a lot of activity and has reached a level of development paralleling the corresponding development in convex optimization complexity. Here the solution operator $S(I,\epsilon)$ typically is defined as $\epsilon$-approximate stationary points or local minima; for instance, in the smooth unconstrained case, $S(I,\epsilon)$ is the set of all points where the gradient norm is at most $\epsilon$. We recommend the recent thesis of Yair Carmon for a fantastic survey, pointers to literature and several new breakthroughs~\cite{carmon}. 

\subsection*{Open Questions}
\begin{enumerate}

\item As mentioned at the beginning of Section~\ref{sec:lenstra}, a major open question in mixed-integer optimization is to bridge the gap of $2^{O(n)}\log(R)$ tight bound on the information complexity and the $2^{O(n\log n)}\log(R)$ algorithmic complexity of Lenstra style algorithms presented in Section~\ref{sec:lenstra}, for pure integer optimization. It seems that radically new ideas are needed and Lenstra style algorithms cannot escape the $2^{O(n\log n)}$ barrier. 

\item Closing the gap between the lower and the upper bounds for $\epsilon$-information complexity in Theorem~\ref{thm:information-comp} for the truly mixed-integer case, i.e., $n,d \geq 1$ seems to be a nontrivial problem. See the discussion after the statement of Theorem~\ref{thm:information-comp} for the two different sources of discrepancy. In particular, Conjecture~\ref{conj:mixed-center} seems to be a very interesting question in pure convex/discrete geometry. It was first stated in Timm Oertel's thesis~\cite{oertel2014integer}, and the thesis and~\cite{basu2017centerpoints} contain some partial results towards resolving it. Further, taking into account the sizes of the responses to oracle queries merits further study. To the best of our knowledge, tight lower and upper bounds are not known when the size of the responses are taken into account, i.e., when one uses the strict definition of $\epsilon$-information complexity as in Definition~\ref{def:comp} of this paper. Finally, we believe that the study of $\epsilon$-information complexity of optimization with respect to oracles that are different from subgradient oracles is worth investigation.

\item The mixed-integer linear optimization problem class (point 2. in Example~\ref{ex:concrete-ex}) can be accessed through two different oracles: the standard algebraic one described in point 2. of Example~\ref{ex:concrete-oracles}, or a separation oracle. Assuming the standard size hierarchy obtained from the binary encoding sizes, it is not hard to implement a separation oracle using the algebraic oracle with polynomially many algebraic oracle calls. However, it is not clear if the algebraic oracle can be implemented in polynomial time by a separation oracle. Essentially, one has to enumerate the facets of the polyhedral feasible region $P\subseteq \R^k$, which is equivalent to enumerating the vertices of the polar $P^\star$ (after finding an appropriate center in $P$ -- see~\cite[Lemma 2.2.6]{lovasz1986algorithmic}). This can be done if we have an optimization oracle for $P^\star$ since one can iteratively enumerate the vertices by creating increasingly better inner polyhedral approximations. Roughly speaking, one maintains the current list of enumerated vertices $v_1, \ldots, v_t$ of $P^\star$ and also an inequality description of $\conv(\{v_1, \ldots, v_t\})$. Now we optimize over $P^\star$ in the direction of all the inequalities describing $\conv(\{v_1, \ldots, v_t\})$. If $\conv(\{v_1, \ldots, v_t\}) \subsetneq P^\star$, then at least one new vertex will be discovered and we can continue the process after reconvexifying. An optimization oracle for the polar $P^\star$ can be implemented in polynomial time because we have a separation oracle for $P$~\cite{sch,lovasz1986algorithmic,GroetschelLovaszSchrijver-Book88}. The only issue is that $\conv(\{v_1, \ldots, v_t\})$ may need $t^k$ inequalities in its description, which is not polynomial unless we consider the dimension $k$ to be fixed. It would be good to resolve whether the separation and algebraic oracles for mixed-integer linear optimization are polynomially equivalent.

\item Theorem~\ref{thm:BC>BB/CP} shows that a complementary pair of branching scheme and cutting plane paradigm can lead to substantial gains when combined into a branch-and-cut algorithm, as opposed to a pure branch-and-bound or pure cutting plane algorithm. Is this a characterization of when branch-and-cut is provably better? In other words, if a branching scheme and cutting plane paradigm are such that there exists a family of instances where branch-and-cut is exponentially better than branch-and-bound or cutting planes alone, then is it true that the branching scheme and the cutting plane paradigm are complementary? A result showing that branch-and-cut is superior if and only if the branching scheme and cutting plane paradigm are complementary would be a tight theoretical characterization of this important phenomenon. While this general question remains open, a partial converse to Theorem~\ref{thm:BC>BB/CP} was obtained in~\cite{basu-BB-CP-II}:

\begin{theorem}\label{thm:conv-BC}~\cite[Theorem 1.14]{basu-BB-CP-II} Let $\D$ be a branching scheme that includes all split disjunctions and let $\CP$ be any cutting plane paradigm. Suppose that for every pure integer instance and any cutting plane proof based on $\CP$ for this instance, there is a branch-and-bound proof based on $\D$ of size at most a polynomial factor (in the size of the instance) larger. Then for any branch-and-cut proof based on $\D$ and $\CP$ for a pure integer instance, there exists a pure branch-and-bound proof based on $\D$ that has size at most polynomially larger than the branch-and-cut proof.
\end{theorem}
\end{enumerate}

\section*{Acknowledgement} The author benefited greatly from discussions with Daniel Dadush at CWI, Amsterdam and Timm Oertel at FAU, Erlangen-N\"urmberg. Comments from two anonymous referees helped the author significantly to consolidate the material, improve its presentation and make tighter connections to the existing literature on the complexity of optimization.

\bibliographystyle{plain}
\bibliography{../../../references/full-bib}
\end{document}